\documentclass[12pt,reqno]{amsart}
\usepackage{amssymb}
\usepackage{diagrams}
\usepackage{graphicx}
\usepackage{amsmath}
\usepackage{a4wide}
\usepackage{version}
\usepackage{mathrsfs}
\usepackage{tikz}
\usepackage{mathtools}
\usetikzlibrary{arrows,decorations.markings, matrix}
\usepackage{hyperref}
\usepackage{float}

\usepackage[all]{xy}

\numberwithin{equation}{section}

\newtheorem{thm}{Theorem}[subsection]
\newtheorem{lem}[thm]{Lemma}
\newtheorem{prop}[thm]{Proposition}
\newtheorem{cor}[thm]{Corollary}

\theoremstyle{definition}
\newtheorem{definition}[thm]{Definition}

\newtheorem{rmk}[thm]{Remark}

\newcommand{\Z}{\mathbb{Z}}
\renewcommand{\k}{\mathbf{k}}

\newcommand{\cA}{\mathcal{A}}

\newcommand{\real}{\operatorname{real}}
\newcommand{\Pf}{\noindent {\it Proof}}
\newcommand{\id}{\operatorname{id}}

\newcommand{\ov}{\overline}

\newcommand{\WW}{{\mathcal W}}

\newcommand{\PP}{{\mathcal P}}

\newcommand{\Hom}{\operatorname{Hom}}
\newcommand{\homs}{\operatorname{hom}}

\newcommand{\Ext}{\operatorname{Ext}}
\newcommand{\End}{\operatorname{End}}
\newcommand{\Endlc}{\operatorname{end}}

\newcommand{\Aut}{\operatorname{Aut}}

\renewcommand{\a}{\alpha}
\renewcommand{\b}{\beta}

\newcommand{\De}{\Delta}

\newcommand{\C}{{\mathbb C}}

\newcommand{\La}{\Lambda}
\newcommand{\Ga}{\Gamma}

\newcommand{\wt}{\widetilde}
\newcommand{\ot}{\otimes}

\newcommand{\sub}{\subset}
\newcommand{\ed}{\qed\vspace{3mm}}

\renewcommand{\k}{\mathbf{k}}

\newcommand{\Id}{\operatorname{Id}}

\renewcommand{\mod}{\operatorname{mod}}

\newcommand{\Sym}{\operatorname{Sym}}

\newcommand{\DD}{{\mathcal D}}

\newcommand{\BB}{{\mathcal B}}

\newcommand{\G}{{\mathbb G}}

\newcommand{\hra}{\hookrightarrow}
\newcommand{\lan}{\langle}
\newcommand{\ran}{\rangle}

\newcommand{\CC}{{\mathcal C}}

\newcommand{\Spec}{\operatorname{Spec}}

\renewcommand{\P}{{\mathbb P}}

\newcommand{\si}{\sigma}

\newcommand{\ga}{\gamma}

\renewcommand{\ker}{\operatorname{ker}}
\newcommand{\im}{\operatorname{im}}

\newcommand{\Perf}{\operatorname{Perf}}

\newcommand{\w}{{\mathbf w}}

\title[Homological mirror symmetry for higher dimensional pairs of pants]{Homological mirror symmetry for higher dimensional pairs of pants}
\author{Yank\i\ Lekili}
\author{Alexander Polishchuk}

\address{King's College London}
\address{University of Oregon, National Research University Higher School of Economics, and
Korea Institute for Advanced Study}


\begin{document}

\begin{abstract}

Using Auroux's description of Fukaya categories of symmetric products of punctured surfaces, we compute the partially wrapped Fukaya category of the complement of $k+1$ generic hyperplanes in $\mathbb{CP}^n$, for $k \geq n$, with respect to certain stops in terms of the endomorphism algebra of a generating set of objects. The stops are chosen so that the resulting algebra is formal.
In the case of the complement of $(n+2)$-generic hyperplanes in $\mathbb{C}P^n$ ($n$-dimensional pair-of-pants), we show that our partial wrapped Fukaya category is equivalent to a certain categorical resolution of the derived category of the singular affine variety $x_1x_2..x_{n+1}=0$. By localizing, we deduce that the (fully) wrapped Fukaya category of the $n$-dimensional 
pair-of-pants is equivalent to the derived category of
$x_1x_2...x_{n+1}=0$. We also prove similar equivalences for finite abelian covers of the $n$-dimensional pair-of-pants.

\end{abstract}

\maketitle

\section{Introduction}

Originally homological mirror symmetry was conceived by Kontsevich as an equivalence of the 
Fukaya category of a compact symplectic manifold with the bounded derived category of coherent sheaves on a mirror
dual compact complex variety. Since then it grew into a vast program connecting Fukaya categories of several kinds associated with not necessarily compact symplectic manifolds with derived categories of several kinds attached to possibly singular
algebraic varieties. 

In \cite{LP}, we proved a version of homological mirror symmetry relating 
Fukaya categories of punctured Riemann surfaces with some derived categories attached to stacky nodal curves.
More precisely, for a punctured Riemann surface  $\Sigma$ and a line field $\eta$ on $\Sigma$, we choose a certain set of stops $\Lambda$, and consider the following sequence of pre-triangulated categories, related by quasi-functors 
\[ \mathcal{F}(\Sigma, \eta) \to \mathcal{W}(\Sigma, \Lambda, \eta) \to \mathcal{W}(\Sigma, \eta) \]
where $\mathcal{F}(\Sigma, \eta)$ is the compact Fukaya category (\cite{seidelbook}), $\mathcal{W}(\Sigma, \Lambda, \eta)$ is the partially wrapped Fukaya category (\cite{aurouxggt}, \cite{HKK}) , and $\mathcal{W}(\Sigma, \eta)$ is the (fully) wrapped Fukaya category (\cite{abouzseidel}). The first functor is full and faithful, and the second functor is a localization functor corresponding to dividing by the full subcategory of Lagrangians supported near $\Lambda$. 

On the mirror side, we consider a nodal stacky curve $C$ obtained by attaching copies of weighted projective lines at their orbifold points (see \cite{LP} for details), and we again have a sequence of categories
\[ \mathrm{Perf}(C) \to D^b(\mathcal{A}_C) \to D^b \mathrm{Coh}(C) \]
where $\mathcal{A}_C$ is a sheaf of algebras, called the Auslander order over $C$ that was previously studied in \cite{BurbanDrozd}. We again have that the first functor is full and faithful, and the second functor is a localization.  
The main result of \cite{LP} is an equivalence of homologically smooth and proper, pre-triangulated categories 
\begin{align}\label{ms} \mathcal{W}(\Sigma, \Lambda, \eta) \simeq D^b(\mathcal{A}_C). \end{align}
It is proved by constructing a generating set of objects on each side and matching their endomorphism algebras. 
The main point is that these algebras turn out to be formal (in fact, concentrated in degree $0$), which means that we
only need to prove an isomorphism of the usual associative algebras and do not have to worry about higher products.

One then deduces an equivalence $\mathcal{W}(\Sigma, \eta) \simeq D^b \mathrm{Coh}(C)$ by identifying the subcategories
on both sides of the equivalence \eqref{ms} with respect to which to take quotient.
The equivalence $\mathcal{F}(\Sigma, \eta) \simeq \mathrm{Perf}(C)$ is deduced by characterizing both sides as subcategories
of the two sides of \eqref{ms}. 
Note that considering the same generators in the localized categories leads to dg-algebras which are far from formal.
Note also that the embedding 
$\mathrm{Perf}(C)\hra D^b(\cA_C)$ is a simple example of categorical resolutions considered in \cite{Kuz-Lunts}. 

Let us explain this in more detail in a simple case. Let $\Sigma$ be the pair-of-pants, that is, a 3-punctured sphere, $\Lambda$ be 2 stops at the outer boundary as drawn in Figure \ref{pop}. We also choose a line field $\eta$ on $\Sigma$ which has rotation number 2 around the outer boundary and 0 along the interior boundary components (see \cite{LPgentle} for a recent detailed study of line fields on punctured surfaces). 

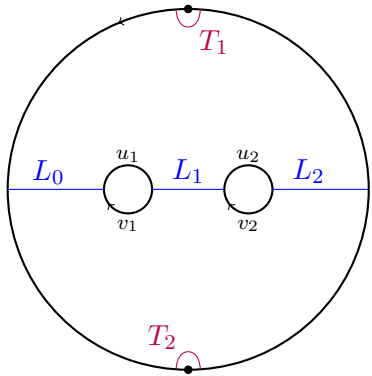
\begin{figure}[ht!]
\centering
\begin{tikzpicture}

\begin{scope}[scale=0.8]

\tikzset{
  with arrows/.style={
    decoration={ markings,
      mark=at position #1 with {\arrow{>}}
    }, postaction={decorate}
  }, with arrows/.default=2mm,
} 

\tikzset{vertex/.style = {style=circle,draw, fill,  minimum size = 2pt,inner        sep=1pt}}
\def \radius {1.5cm}

\foreach \s in {1,2,3,4,5,6,7,8,9,10} {
    \draw[thick] ([shift=({360/10*(\s)}:\radius-1.1cm)]-1,0) arc ({360/10 *(\s)}:{360/10*(\s+1)}:\radius-1.1cm);
    \draw[thick] ([shift=({360/10*(\s)}:\radius-1.1cm)]1,0) arc ({360/10 *(\s)}:{360/10*(\s+1)}:\radius-1.1cm);
   
   \draw[thick] ([shift=({360/10*(\s)}:\radius+1.5cm)]0,0) arc ({360/10 *(\s)}:{360/10*(\s+1)}:\radius+1.5cm);
}

\draw[with arrows] ([shift=({360/10*(7)}:\radius-1.1cm)]-1,0) arc ({360/10 *(7)}:{360/10*(5)}:\radius-1.1cm);
\draw[with arrows] ([shift=({360/10*(7)}:\radius-1.1cm)]1,0) arc ({360/10*(7)}:{360/10*(5)}:\radius-1.1cm);
   
\draw[with arrows]({360/10 * (3)}:\radius+1.5cm) arc ({360/10 *(3)}:{360/10*(4)}:\radius+1.5cm);

\draw [blue] ([shift=({360/10*(5)}:\radius-1.1cm)]-1,0) -- ([shift=({360/10*(5)}:\radius+1.5cm)]0,0);  
\draw [blue] ([shift=({360/10*(10)}:\radius-1.1cm)]1,0) -- ([shift=({360/10*(10)}:\radius+1.5cm)]0,0);

\draw [blue] ([shift=({360/10*(10)}:\radius-1.1cm)]-1,0) --  ([shift=({360/10*(5)}:\radius-1.1cm)]1,0);

\node[blue] at ([shift=({360/10*(4)}:\radius-1.1cm)]-2,0.05) {\small $L_{0}$}; 

\node[blue] at (0,0.3) {\small $L_{1}$};
\node[blue] at (2,0.3) {\small $L_{2}$};




\node[vertex] at  ([shift=({360/10*(7.5)}:\radius+1.5cm)]0,0) {};
\node[vertex] at  ([shift=({360/10*(2.5)}:\radius+1.5cm)]0,0) {};

\node at (-1,0.55){\tiny $u_1$};
\node at (-1,-0.6){\tiny $v_1$};

\node at (1,0.55){\tiny $u_2$};
\node at (1,-0.6){\tiny $v_2$};

\draw[purple] (-0.2, 3) to[in=180,out=270] (0,2.7);
\draw[purple] (0, 2.7) to[in=270,out=0] (0.2,3);

\draw[purple] (-0.2, -3) to[in=180,out=90] (0,-2.7);
\draw[purple] (0, -2.7) to[in=90,out=0] (0.2,-3);

\node[purple] at ([shift=({360/10*(7)}:\radius+1.5cm)]0.5,0.4) {\small $T_2$}; 
\node[purple] at ([shift=({360/10*(2)}:\radius+1.5cm)]-0.5,-0.4) {\small $T_1$};

\end{scope}

\end{tikzpicture}
    \caption{Pair-of-pants}
    \label{pop}
\end{figure}

The partially wrapped Fukaya category $\mathcal{W}(\Sigma, \Lambda, \eta)$ is generated by the Lagrangians $L_0, L_1, L_2$ drawn on Figure \ref{pop}, and their endomorphism algebra is easily computed to be given by the quiver with relations
on Figure \ref{quiverpop}.

\begin{figure}[H]
\begin{tikzpicture}
    \tikzset{vertex/.style = {style=circle,draw, fill,  minimum size = 2pt,inner    sep=1pt}}

        \tikzset{edge/.style = {->,-stealth',shorten >=8pt, shorten <=8pt  }}

\begin{scope}[scale=0.8]
\node[vertex] (a) at  (0,0) {};
\node[vertex] (a1) at (2.5,2) {};
\node[vertex] (a3) at (5,0) {};

\node at (-0.5,0.3){\tiny $L_0$};
\node at (2.5,2.3) {\tiny $L_1$};
\node at (5.5,0.3) {\tiny $L_2$};

\node at (0.8, 1.5) {\tiny $u_1$};
\node at (1.7, 0.7) {\tiny $v_1$};
\node at (3.2, 0.7) {\tiny $v_2$};
\node at (4.2, 1.5) {\tiny $u_2$};


\draw[edge] (a)  to[out=55,in=190]  (a1);
\draw[edge] (a1)  to[out=230,in=20]  (a);
\draw[edge] (a1) to[out=-10, in=120] (a3);
\draw[edge] (a3) to[out=160, in=310] (a1);

\node at (9, 1) {$u_2 u_1 = 0 = v_1 v_2$};

\end{scope} 

\end{tikzpicture}
	\caption{Endomorphism algebra of a generating set}
\label{quiverpop}
\end{figure}
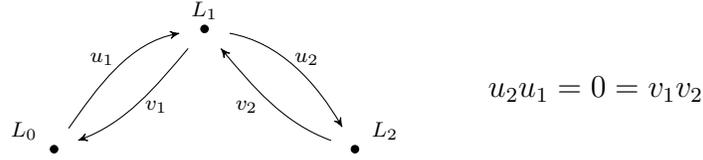

On the B-side, the mirror is given by the Auslander order $A$ over the node algebra
\[ R = \mathbf{k}[x_1,x_2] / (x_1 x_2), \]
where $\k$ is a commutative ring.               
The Auslander order in this case is simply,
\[ A = End_R ( R/(x_1) \oplus R/(x_2) \oplus R ). \] 
One can directly see that $A$ is isomorphic to the quiver algebra given in Figure \ref{quiverpop}.
Note that $R$ is a Cohen-Macaulay algebra and the modules $R$, $R/(x_1)$, $R/(x_2)$ 
                comprise the set of indecomposable maximal Cohen-Macaulay modules of $R$, 
              and they generate $D^b(A)$ as a triangulated category, provided $\k$ is regular.

Now, the wrapped Fukaya category $\mathcal{W}(\Sigma, \eta)$ is the localization of $\mathcal{W}(\Sigma, \Lambda, \eta)$ given by dividing out by the subcategory generated by the objects $T_1, T_2$ supported near the stops. We can express them in terms of $L_0, L_1, L_2$ as follows:
\[ T_1 \simeq \{ L_0 \xrightarrow{u_1} L_1 \xrightarrow{u_2} L_2 \} \] 
\[ T_2 \simeq \{ L_2 \xrightarrow{v_2} L_1 \xrightarrow{v_1} L_0 \} \] 
Similarly, $D^b(R)$ is the localization of $D^b(A)$ obtained by dividing out by the corresponding subcategory, and this allows one to establish an equivalence
\[ \mathcal{W}(\Sigma, \eta) \simeq D^b(R).\]

\subsection{New results} 

In this paper, we apply the above strategy to prove homological mirror symmetry for the higher-dimensional pair-of-pants,
\[ \mathcal{P}_n = \mathrm{Sym}^n ( \mathbb{P}^1 \setminus \{p_0, p_1, \ldots, p_{n+1}\} ), \] 
where $\mathrm{Sym^n}(\Sigma) = \Sigma^{n} / \mathfrak{S}_n$, 
see Section \ref{review} for a brief review of symplectic topology of these spaces. In other words, $n$-dimensional pair-of-pants is the complement of $(n+2)$ generic hyperplanes in $\mathbb{P}^n$.

On the A-side, we first introduce a stop $\Lambda = \Lambda_{1} \cup \Lambda_{2}$ where $\Lambda_{i} = \{ q_i \} \times \mathrm{Sym}^{n-1} (\Sigma)$ for some base points $q_1, q_2$. We pick a grading structure $\eta$ and 
consider the partially wrapped Fukaya $\mathcal{W}(\mathcal{P}_n, \Lambda, \eta)$, where we use 
some commutative ring $\mathbf{k}$ as coefficients. There is a natural
generating set of objects $\{ L_S : S \subset \{0,1,\ldots, n+1\}, |S|=n \}$ in this category
and our first result is an explicit computation of the algebra of morphisms between these objects. 
In fact, we do this more generally for the symplectic manifolds
\[ M_{n,k} = \mathrm{Sym}^n (\mathbb{P}^1 \setminus \{ p_0,p_1,\ldots, p_k \}), \]
see Theorem \ref{A-side-morphisms-general-thm} for the precise description of the resulting algebra.

Next, we specialize to the case $k=n+1$.
In this case, on the B-side we consider the categorical resolutions of the algebra
\[ R = \mathbf{k}[x_1,x_2,\ldots, x_{n+1}] / (x_1 x_2 \ldots x_{n+1}) \]
given by
\[ \BB^{\circ} :=\End_R(R/(x_1)\oplus R/(x_{[1,2]})\oplus \ldots \oplus R/(x_{[1,n]})\oplus R) \]
and
\[ \BB^{\circ\circ}:=\End_R(\bigoplus_{I\sub [1,n+1],I\neq\emptyset}R/(x_I)), \] 
where the summation is over all nonempty subintervals of $[1,n+1]$, and for $I \subset [1,n+1]$ we use the notation $x_{I} : = \prod_{i \in I} x_i$.  

By explicit computations we prove the following theorem (see Theorem \ref{main-thm}).

\begin{thm} 
There exists a grading structure $\eta$ on $\mathcal{P}_n$ such that we have an equivalence of pre-triangulated categories over $\mathbf{k}$, 
\[ \mathcal{W}(\mathcal{P}_n , \Lambda,\eta) \simeq \Perf(\BB^{\circ\circ}).  \]
\end{thm}

We next analyze the localization of these categories corresponding to dividing out by objects supported near the stops to deduce homological mirror symmetry for the (fully) wrapped Fukaya categories:

\begin{cor} \label{popproof}  For the same grading structure $\eta$ we have
equivalences of pre-triangulated categories over $\mathbf{k}$: 
\[\mathcal{W}(\mathcal{P}_n, \Lambda_{1},\eta) \simeq \Perf(\BB^{\circ}).\]
Assume that $\mathbf{k}$ is a regular ring. Then we also have an equivalence                                                                         
\[\mathcal{W}(\mathcal{P}_n,\eta) \simeq D^b \mathrm{Coh}( x_1 x_2 \ldots x_{n+1}=0 ). \]
\end{cor}
                                                                         
Finally, we deduce similar equivalences for finite abelian covers of $\PP_n$, associated with homomorphisms
$\pi_1(\PP_n)\simeq \Z^{n+1}\to\Gamma$ into finite abelian groups $\Gamma$. On the B-side we take equivariant
versions of the above categories with respect to the natural actions of the dual finite commutative group 
                                                                         scheme $G=\Gamma^*$
(see Theorem \ref{abelian-cover-thm}).                                                                         

Note that to prove the homological mirror symmetry statement (the second equivalence of Corollary \ref{popproof}), it is enough to work with one stop $\Lambda_1$.
The reason we consider the picture with two stops is due to the relation with the Auslander order discussed above and to Ozsv\'ath-Szab\'o's bordered algebras 
(see also Remarks \ref{1dim-coverings-rem} and \ref{OS-rem}).                                                                         
                                                                         
\subsection{Relation to other works} 

Homological mirror symmetry for pair-of-pants is a much studied subject. However, a complete proof of Corollary \ref{popproof} has not appeared in writing until this paper. A motivation for studying these particular examples of mirror symmetry comes from a theorem of Mikhalkin \cite{mikhalkin} that a hypersurface in $\mathbb{C}P^{n+1}$ admits a decomposition into several $\mathcal{P}_n$, much like a Riemann surface admits a decomposition into several $\mathcal{P}_1$. 

In \cite{sheridan}, Sheridan identifies the mirror (immersed) Lagrangian in $\mathcal{P}_n$ corresponding to $\mathcal{O}_0$, the structure sheaf of the origin in the triangulated category of singularities of the normal crossing divisor $x_0x_1x_2\ldots x_{n+1}=0$ in $\mathbb{C}^{n+2}$. 
By a theorem of Orlov \cite{Orlov2004}, the latter category is quasi-equivalent to the matrix factorization category $\mathrm{mf}(\mathbb{C}^{n+2}, x_0x_1x_2 \ldots x_{n+1})$. Note that by the result of Isik \cite{Isik}, the latter category (or more precisely, its $\mathbb{G}_m$-equivariant version where $x_1,..,x_{n+1}$ have weight 0, and $x_{0}$ has weight 2) is naturally quasi-equivalent to the derived category of 
$x_1x_2\ldots x_{n+1}=0$ (see also \cite{Shipman}). 
Under this equivalence, $\mathcal{O}_0$ corresponds to a perfect object supported at the origin. 

In \cite{nadler} instead of wrapped Fukaya category of $\mathcal{P}_n$, Nadler studies the $\mathbb{Z}_2$-graded category of wrapped microlocal sheaves associated to a skeleton of $\mathcal{P}_n$ (see also the follow-up paper by Gammage and Nadler \cite{GamNad}). It is then verified that this category agrees with the $\mathbb{Z}_2$-graded category of matrix factorizations $\mathrm{mf}(\mathbb{C}^{n+2},x_0x_1x_2\ldots x_{n+1})$. It is expected and in certain cases proved that wrapped microlocal sheaves category is equivalent to wrapped Fukaya category - see \cite{GPS3} for the case of cotangent bundles. However, such an equivalence for arbitrary Weinstein manifold is not yet accomplished. Nonetheless, in view of the works of Ganatra-Pardon-Shende \cite{GPS1}, \cite{GPS2},\cite{GPS3} such an equivalence in the case of $\mathcal{P}_n$ seems to be within reach (see in particular the discussion in \cite[Section 6.6]{GPS3} which outlines a proof depending on a work-in-progress). Establishing such an equivalence for $\mathcal{P}_n$ would give another confirmation for Corollary \ref{popproof} (at least, in the $\mathbb{Z}_2$-graded case).

The closest to our work is \cite{aurouxspec}, where Auroux sketches a proof of homological mirror symmetry for $\mathcal{P}_n$ which depends on certain conjectures about generation by an explicit collection of Lagrangians and a classification of 
$A_\infty$-structures on their cohomology.

We also mention that partially wrapped Fukaya categories of symmetric products appear predominantly in Heegaard Floer homology \cite{LOT}, \cite{aurouxggt}, \cite{aurouxicm} (see also \cite{LPer}). In particular, our computations of $\mathcal{W}(M_{n,k})$ give an alternative viewpoint for knot Floer homology. We defer this to a future work.

The paper is organized as follows. After reviewing some background material on partially wrapped Fukaya categories
of symmetric powers of Riemann surfaces in Sec.\ \ref{review}, we present the computation of the algebra of morphisms between
generating Lagrangians in $M_{n,k}$ in Sec.\ \ref{A-sec}. Then in Sec.\ \ref{B-sec} we deal with the B-side of the story:
we study the derived categories of modules over $\BB^{\circ}$ and $\BB^{\circ\circ}$. In particular,
we construct semiorthogonal decompositions of these categories and obtain the localization results similar to those
corresponding to the removal of a stop on the A-side.
Finally, in Sec.\ \ref{comparison-sec} we prove the equivalences of categories on the A-side and on the B-side
(see Theorem \ref{main-thm}).

\medskip
                                                                         
\noindent 
{\it Conventions}. We work over a base commutative ring $\mathbf{k}$.
When we write complexes of modules in the form $[\ldots \to \cdot]$, we assume that the rightmost
term sits in degree $0$. The morphism complexes in dg-categories are denoted by lowercase $\mathrm{hom}$,
while their cohomology are denote by $\Hom^\bullet$. We also abbreviate $\Hom^0$ simply as $\Hom$. 
The wrapped Fukaya category we consider is defined as the split-closed pre-triangulated envelope of the category whose objects are graded exact Lagrangians.

\medskip
                                                                         
\noindent
{\it Acknowledgments}. A.P. is grateful to Victor Ostrik for useful discussions.
Y.L. is partially supported by the Royal Society (URF). A.P. is partially supported by the NSF grant DMS-1700642, 
by the National Center of Competence in Research ``SwissMAP - The Mathematics of Physics'' of the Swiss National Science Foundation, and by the Russian Academic Excellence Project `5-100'.
While working on this project, A.P. was visiting Korea Institute for Advanced Study and ETH Zurich. 
He would like to thank these institutions for hospitality and excellent working conditions.

\section{A brief review of Fukaya categories of symmetric products of Riemann surfaces}
\label{review}
Let $\Sigma$ be a Riemann surface, for each $n>0$, there exists a smooth $n$-dimensional complex algebraic variety 
\begin{align} 
	\Sym^n (\Sigma) := \Sigma^{n} / \mathfrak{S}_n
\end{align}
where $\mathfrak{S}_n$ is the permutation group which acts by permuting the components of the product. 

Let $\pi : \Sigma^{n} \to \Sym^n (\Sigma)$ be the branched covering map. Fix an area form $\omega$ on $\Sigma$. In \cite[Section 7]{perutz}, Perutz explains how to smoothen the closed current $\pi_*(\omega^{\times n})$ on $\Sym^n(\Sigma)$ to a K\"ahler form $\Omega$ by modifying it in an arbitrarily small analytic neighborhood of the (big) diagonal. In particular, outside this neighborhood we have $\Omega = \pi_* (\omega^{\times n})$. 
Throughout, we will view $\Sym^n(\Sigma)$ as a symplectic manifold equipped with such a K\"ahler form $\Omega$. 

If we write $g = g(\Sigma)$ for the genus of $\Sigma$, the first Chern class of such a variety is given by 
\[ c_1( \Sym^n(\Sigma)) = (n+1-g) \eta  - \theta \]
where $\eta$ and $\theta$ are the Poincar\'e duals of the class $\{pt\} \times \Sym^{n-1}(\Sigma)$ and the theta divisor, respectively. These two cohomology classes span the invariant part of $H^2(\Sym^n(\Sigma))$ under the action of the mapping class group of $\Sigma$. Moreover, we have that $[\Omega]= \eta$.

In particular, when \begin{align} \Sigma = \P^1 \setminus \{p_0, p_1,\ldots, p_k\} \end{align} $\Sym^n(\Sigma)$ is an exact symplectic manifold with $c_1=0$. Such symplectic manifolds are sometimes referred to as symplectically Calabi-Yau manifolds, and their Fukaya categories can be $\Z$-graded \cite{seidelgraded}. From the point of view of symplectic topology, the positions of the points do not matter, so let us introduce the notation
	\begin{align}
		M_{n,k} = \Sym^{n}(\P^1 \setminus \{p_0, p_1,\ldots, p_k\}) 
	\end{align}
to denote the exact symplectic manifold with $c_1=0$. The grading structures are given by homotopy classes of trivializations 
of the bicanonical bundle, and there is an effective $H^1(\Sym^n(\Sigma)) \simeq H^1(\Sigma)$ worth of choices.

Recall the well-known isomorphism of algebraic varieties
\begin{align}
     \Sym^n ( \P^1) \simeq \P^n 
\end{align}
given by sending 
an effective divisor of degree $n$ on $\P^1$ to its homogeneous equation defined up to rescaling.

Therefore, one can think of $M_{n,k}$ as the complement of $k+1$ generic hyperplanes in $\P^n$. This provides an alternative way to equip $M_{n,k}$ with a symplectic structure by viewing it as an affine variety but we will not pursue this any further, as we prefer to emphasize the structure of $M_{n,k}$ as a symmetric product on a punctured genus 0 surface. The two symplectic structures are equivalent as they both tame the standard complex structure $J=\mathrm{Sym}^n(j)$ on $M_{n,k}$ induced from $\mathbb{P}^n$ (see \cite[Prop. 1.1]{perutz}). This also makes it clear that for $0 \leq k < n$, $M_{n,k} = \mathbb{C}^{n-k} \times \mathbb{(C^*)}^{k}$ which is a subcritical Stein manifold, so our main interest will be for $k \geq n$. 

We will also equip $M_{n,k}$ with stops $\Lambda_{Z}$ corresponding to choice of symplectic hypersurfaces of the form $\{p\} \times \Sym^{n-1}(\Sigma)$ for $p \in Z$, where $Z$ is finite set of points. The set $Z$ will be indicated by choosing stops in the ideal boundary of $\Sigma$. More precisely, by removing cylindrical ends, we view $\Sigma$ as a 2-dimensional surface with boundary and the set $Z$ will be chosen as a finite set of points on $\partial \Sigma$. 

We write $\mathcal{W} (M_{n,k}, \Lambda_Z)$ for the partially wrapped Fukaya
category. Motivated by bordered Heegaard Floer homology \cite{LOT}, these categories were originally constructed by Auroux in the papers
\cite{aurouxggt}, \cite{aurouxicm}. These works provide foundational results on these categories,
as well as some very useful results about generating objects and
existence of certain exact triangles. 

All of the Lagrangians that we use will be of the form $L_1 \times L_2 \times \ldots L_n$ where $L_i \subset \Sigma$ are pairwise disjoint Lagrangian arcs in $\Sigma$, which can be considered as objects in $\mathcal{W}(\Sigma, Z)$. 
Auroux proves in \cite[Theorem 1]{aurouxicm} that given a set of Lagrangians $L_0, L_1,\ldots, L_k$ such that their complement in $\Sigma$ is a disjoint union of disks with at most 1 stop in their boundary then for $1 \leq n \leq (k+1)$, the corresponding partially wrapped Fukaya category of $\Sym^n(\Sigma)$ is generated by $k+1 \choose n$ product Lagrangians $L_{i_1} \times \ldots L_{i_n}$ where $(i_1,\ldots, i_n)$ runs through size $n$ subsets of $\{0,1,\ldots, k\}$. Notice that this generation result only depends on the configuration of $L_i$ on $\Sigma$ and is independent of $n$.

Furthermore, Auroux explains how to compute the dg-endomorphism algebra for such product Lagrangians 
(see \cite[Prop.\ 11]{aurouxicm}; note that there are no higher products). As vector spaces, the morphism spaces are defined by
\begin{align*} &\mathrm{hom}(L_{i_1} \times L_{i_2} \times \ldots \times L_{i_n}, L_{j_1} \times L_{j_2} \times \ldots \times L_{j_n}) \\ &= \bigoplus_{\sigma} \mathrm{hom}(L_{i_1}, L_{\sigma(i_1)}) \otimes \mathrm{hom}(L_{i_2}, L_{\sigma(i_2)}) \otimes \ldots \otimes \mathrm{hom}(L_{i_n}, L_{\sigma(i_n)})  \end{align*}
where $\sigma$ runs through bijections $\{i_1,\ldots, i_n\}\to\{j_1, \ldots, j_n\}$.

Following \cite{LOT}, we can represent these generators via strand diagrams as follows. 
First, the endpoints of the Lagrangians $L_{i_1},\ldots, L_{i_n}$, $L_{j_1},\ldots, L_{j_n}$ are grouped into equivalence classes according to which boundary component of $\Sigma$ they end on. Note that the set of endpoints lying on each boundary component has a cyclic order induced by the orientation of the boundary. Thus, given a morphism $(f_1,\ldots f_n)$, we can assume that it is of the form $(f_{i_1,1}, \ldots, f_{i_{r_1},1},f_{i_1,2}, \ldots, f_{i_{r_2},2} \ldots, f_{i_1,k},\ldots, f_{i_{r_k},k})$
where $f_{i_1,s},\ldots, f_{i_{r_s},s}$ are Reeb chords along the $s^{th}$ boundary component 
$\partial\Sigma_s\sub \partial\Sigma$, where the Reeb flow is simply the rotation along the orientation of the boundary. 
Thus, 
each $f_{i_j,s}$ either represents the idempotent of the corresponding Lagrangian $L_{i_j,s}$ or goes in the strictly positive direction along $\partial\Sigma_s$. Thus, the set of Reeb chords  $f_{i_1,s},\ldots, f_{i_{r_s},s}$ can be represented in $\mathbb{R} \times [0,1]$ as upward veering strands from $\mathbb{R} \times \{0\}$ to $\mathbb{R} \times \{1\}$, or as a straight horizontal line if it corresponds to an idempotent. Here, $\mathbb{R}$ is the universal cover of the component $\partial\Sigma_s$
 and $[0,1]$ is the time direction, see Figure \ref{figure0}.

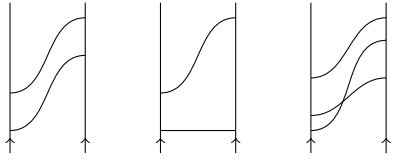
\begin{figure}[H]
\centering
\begin{tikzpicture}

\tikzset{
  with arrows/.style={
    decoration={ markings,
      mark=at position #1 with {\arrow{>}}
    }, postaction={decorate}
  }, with arrows/.default=2mm,
} 

\tikzset{vertex/.style = {style=circle,draw, fill,  minimum size = 2pt,inner        sep=1pt}}
\def \radius {1.5cm}

\draw[with arrows] (0,0) -- (0,2); 
\draw[with arrows] (1,0) -- (1,2); 

\draw (0,0.8) to[in=180, out=0] (1, 1.8);
\draw (0,0.3) to[in=180, out=0] (1, 1.3);

\begin{scope}[xshift=2cm]
\draw[with arrows] (0,0) -- (0,2); 
\draw[with arrows] (1,0) -- (1,2); 

\draw (0,0.8) to[in=180, out=0] (1, 1.8);
\draw (0,0.3) to[in=180, out=0] (1, 0.3);
\end{scope}

\begin{scope}[xshift=4cm]
\draw[with arrows] (0,0) -- (0,2); 
\draw[with arrows] (1,0) -- (1,2); 

\draw (0,1) to[in=180, out=0] (1, 1.8);
\draw (0,0.3) to[in=180, out=0] (1, 1.5);
\draw (0,0.5) to[in=180, out=0] (1, 1);
\end{scope}

\end{tikzpicture}
    \caption{A strand diagram with 3 boundary components}
    \label{figure0}
\end{figure}

In the case when a boundary component contains stops
the Reeb chords are not allowed to pass through the stops.
Hence, instead of using the universal cover $\mathbb{R}$, one cuts along the stops and uses the subintervals to draw the strand diagram. We do not elaborate on the notation to describe this.

Now, the product in $\mathcal{W}(M_{n,k}, \Lambda_Z)$ is induced by the composition in $\mathcal{W}(\Sigma, Z)$. Namely, we have
\[ (f_1 \otimes \ldots \otimes f_n) \circ (g_1 \otimes \ldots \otimes g_n) = (f_1 g_{\sigma(1)} \otimes f_2 g_{\sigma(2)} \otimes \ldots \otimes f_n g_{\sigma(n)}) \]
if there exists a $\sigma \in \mathfrak{S}_n$ such that all the compositions $f_i g_{\sigma(i)}$ in $\mathcal{W}(\Sigma, Z)$ are non-zero, and with the additional important condition that in the strand representation no two strands of the concatenated diagram cross more than once; otherwise the product is set to be zero, see Figure \ref{notallow}
\begin{figure}[ht!]
\centering
\begin{tikzpicture}

\tikzset{
  with arrows/.style={
    decoration={ markings,
      mark=at position #1 with {\arrow{>}}
    }, postaction={decorate}
  }, with arrows/.default=2mm,
} 

\tikzset{vertex/.style = {style=circle,draw, fill,  minimum size = 2pt,inner        sep=1pt}}
\def \radius {1.5cm}

\draw[with arrows] (0,0) -- (0,2); 
\draw[with arrows] (1,0) -- (1,2); 

\draw (0,0.3) to[in=200, out=0] (0.6, 1.4);
\draw (0,0.5) to[in=240, out=0] (0.6, 1);
\draw (0.6,1.4) to[in=180, out=20] (1, 1.5);
\draw (0.6,1) to[in=180, out=60] (1, 1.8);

\end{tikzpicture}
    \caption{A strand diagram with two strands crossing more than once.}
    \label{notallow}
\end{figure}
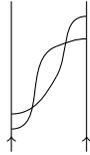

The differential on the space of morphisms is defined as the sum of all the ways of resolving one
crossing of the strand diagram excluding resolutions in which two strands intersect
twice, see Figure \ref{diffex}.

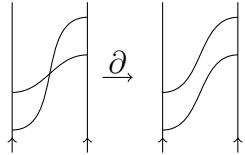
\begin{figure}[ht!]
\centering
\begin{tikzpicture}

\tikzset{
  with arrows/.style={
    decoration={ markings,
      mark=at position #1 with {\arrow{>}}
    }, postaction={decorate}
  }, with arrows/.default=2mm,
} 

\tikzset{vertex/.style = {style=circle,draw, fill,  minimum size = 2pt,inner        sep=1pt}}
\def \radius {1.5cm}

\draw[with arrows] (0,0) -- (0,2); 
\draw[with arrows] (1,0) -- (1,2); 

\draw (0,0.8) to[in=180, out=0] (1, 1.3);
\draw (0,0.3) to[in=180, out=0] (1, 1.8);

\draw[->] (1.2, 1) -- (1.6,1);
\node at (1.4,1.2) {$\partial$};

\begin{scope}[xshift=2cm]
\draw[with arrows] (0,0) -- (0,2); 
\draw[with arrows] (1,0) -- (1,2); 

\draw (0,0.8) to[in=180, out=0] (1, 1.8);
\draw (0,0.3) to[in=180, out=0] (1, 1.3);
\end{scope}

\end{tikzpicture}
    \caption{Resolution of strand diagram}
    \label{diffex}
\end{figure}

In this way we get an explicit dg-category, quasi-equivalent to $\WW(M_{n,k},\Lambda_Z)$.
In what follows, we use these results without further explanations. 

We can choose a line field to give $\mathcal{W}(\Sigma, Z)$ a $\mathbb{Z}$-grading (see \cite{LPgentle} for a recent study of this structure). There are effectively $H^1(\Sigma)$ worth of choices for the line field. The set of grading structures for $M_{n,k}$ is a torsor for an isomorphic group $H^1(\Sym^n(\Sigma)) \cong H^1(\Sigma)$. 
However, the relation between grading structures on $\Sigma$ and on $\Sym^n(\Sigma)$ seems to be quite subtle:
it is easy to see that the grading of a morphism $(f_1 \otimes \ldots f_n)$ in $\mathcal{W}(M_{n,k}, \Lambda_Z)$ cannot be given by the sum of the gradings of morphisms $f_i$ in $\mathcal{W}(\Sigma, Z)$. For example, we will encounter objects $L_1, L_2$ and morphisms $u \in \homs(L_1, L_2)$ and $v \in \homs(L_2, L_1)$ such that \[ \partial(\id_{L_1} \otimes uv) = \partial (vu \otimes \id_{L_2}) = v \otimes u \in \homs(L_1 \times L_2, L_1 \times L_2). \]
This makes direct determination of the gradings in $\mathcal{W}(M_{n,k}, \Lambda_Z)$ difficult. 
Instead, we are able to pin down the grading structures on $M_{n,k}$ using an explicit calculation of the endomorphism
algebra of a generating set of Lagrangians. 


Finally, we recall a basic exact triangle from \cite[Lemma 5.2]{aurouxggt}. Let us consider the Lagrangians
$L = L_1 \times L_2 \times \ldots \times L_n$, $L' = L'_1 \times L_2 \times \ldots L_n$, and 
$L'' = L''_1 \times L_2 \times \ldots L_n$, where $L''_1$ is the arc obtained by sliding $L_1$ along $L'_1$. Then $L$, $L'$ and $L''$ fit into an exact triangle
\begin{equation} \label{exacttri}
L\rTo{u\otimes\id} L'\to L''\to L[1]
\end{equation}
coming from an exact triangle
$$L_1\rTo{u} L'_1\to L''_1\to L_1[1]$$
in $\mathcal{W}(\Sigma, Z)$ 

Similarly, if $L= L_1 \times L_2 \times \ldots L_n$ and $L' = L_1' \times L_2 \times \ldots L_n$, where $L_1'$ is obtained by 
sliding $L_1$ along $L_2$, then $L$ and $L'$ are isomorphic in the category $\mathcal{W}(M_{n,k}, \Lambda_Z)$. 
Indeed, in this situation one can show that $L$ and $L'$ are Hamiltonian isotopic (see \cite{perutz}, \cite{aurouxggt}).

\section{A-side}\label{A-sec}

Throughout, we will work over a commutative ring $\mathbf{k}$. We consider the sphere $\Sigma_{k}$ with $(k+1)$ holes and 2 stops $Z = q_1 \cup q_2$ on one of the boundary components. We have a generating set of Lagrangians $L_{0}$, $L_{1}, \ldots, L_{k}$, which connect $i^{th}$ hole to $(i+1)^{th}$ hole for $i \in \mathbb{Z}/(k+1)$, see Figure \ref{puncsphere} for $k=3$. As in Figure \ref{puncsphere}, we view $\Sigma_k$ as a $k$-holed disk. We call the punctures that lie in the interior of the disk the \emph{interior punctures} of $\Sigma_k$ and label them with $1,2,\ldots, k$ from left to right. We call the unique puncture that corresponds to the boundary of the disk, the \emph{exterior puncture} of $\Sigma_k$ and label it with 0. 

Let $M_{n,k} = \Sym^{n}(\Sigma_{k})$ and $\La=\Lambda_Z = \Lambda_{1} \cup \Lambda_{2}$ be the corresponding stops. Thus, $\Lambda_{i} = q_i \times \Sym^{n-1}(\Sigma_{k})$ are symplectic hypersurfaces in $M_{n,k}$.

The objects $L_0,\ldots,L_k$ generate the partially wrapped Fukaya category $\WW(\Sigma,Z)$. 
Furthermore, by Auroux's theorem \cite[Theorem 1]{aurouxicm}, the category
$\mathcal{W}(M_{n,k}, \Lambda)$ is generated by the Lagrangians 
\[ L_S = L_{i_1} \times L_{i_2} \times \cdots \times L_{i_n} \]
where $i_j \in S$ and $S$ is a subset of $[0,k]$ of size $n$.

Below we are going to describe the algebra 
\begin{align} \mathcal{A}^{\circ\circ} = \bigoplus_{S, S'} \Hom_{\mathcal{W}(M_{n,k}, \Lambda)} (L_S, L_{S'}) \end{align}
It will turn out that for $n<k$, $\cA^{\circ\circ}$ is in fact an $R$-algebra, where
\begin{align} 
R = \mathbf{k}[x_1,\ldots, x_k] /(x_1\ldots x_k).
\end{align} 
Here $x_i$ will correspond to the closed Reeb orbit around the $i^{th}$ interior puncture of $\Sigma_{k}$. 


\begin{figure}[ht!]
\centering
\begin{tikzpicture}

\tikzset{
  with arrows/.style={
    decoration={ markings,
      mark=at position #1 with {\arrow{>}}
    }, postaction={decorate}
  }, with arrows/.default=2mm,
} 

\tikzset{vertex/.style = {style=circle,draw, fill,  minimum size = 2pt,inner        sep=1pt}}
\def \radius {1.5cm}

\foreach \s in {1,2,3,4,5,6,7,8,9,10} {
   \draw[thick] ([shift=({360/10*(\s)}:\radius-1.1cm)]0,0) arc ({360/10 *(\s)}:{360/10*(\s+1)}:\radius-1.1cm);
    \draw[thick] ([shift=({360/10*(\s)}:\radius-1.1cm)]-1.5,0) arc ({360/10 *(\s)}:{360/10*(\s+1)}:\radius-1.1cm);
    \draw[thick] ([shift=({360/10*(\s)}:\radius-1.1cm)]1.5,0) arc ({360/10 *(\s)}:{360/10*(\s+1)}:\radius-1.1cm);
   
   \draw[thick] ([shift=({360/10*(\s)}:\radius+1.5cm)]0,0) arc ({360/10 *(\s)}:{360/10*(\s+1)}:\radius+1.5cm);
}

\draw[with arrows] ([shift=({360/10*(7)}:\radius-1.1cm)]0,0) arc ({360/10*(7)}:{360/10*(6)}:\radius-1.1cm);
\draw[with arrows] ([shift=({360/10*(7)}:\radius-1.1cm)]-1.5,0) arc ({360/10 *(7)}:{360/10*(6)}:\radius-1.1cm);
\draw[with arrows] ([shift=({360/10*(7)}:\radius-1.1cm)]1.5,0) arc ({360/10*(7)}:{360/10*(6)}:\radius-1.1cm);
   
\draw[with arrows]({360/10 * (3)}:\radius+1.5cm) arc ({360/10 *(3)}:{360/10*(4)}:\radius+1.5cm);

\draw [blue] ([shift=({360/10*(5)}:\radius-1.1cm)]-1.5,0) -- ([shift=({360/10*(5)}:\radius+1.5cm)]0,0);  
\draw [blue] ([shift=({360/10*(10)}:\radius-1.1cm)]1.5,0) -- ([shift=({360/10*(10)}:\radius+1.5cm)]0,0);

\draw [blue] ([shift=({360/10*(10)}:\radius-1.1cm)]-1.5,0) --  ([shift=({360/10*(5)}:\radius-1.1cm)]0,0); 
\draw [blue] ([shift=({360/10*(10)}:\radius-1.1cm)]0,0) -- ([shift=({360/10*(5)}:\radius-1.1cm)]1.5,0);  

\node[blue] at ([shift=({360/10*(4)}:\radius-1.1cm)]-2,0.05) {\small $L_{0}$}; 
\node[blue] at ([shift=({360/10*(10)}:\radius+0.7cm)]0.3,0.25) {\small $L_{3}$}; 

\node[blue] at (-0.8,0.3) {\small $L_{1}$};
\node[blue] at (0.7,0.3) {\small $L_{2}$};

\draw[purple] (-0.2, 3) to[in=180,out=270] (0,2.7);
\draw[purple] (0, 2.7) to[in=270,out=0] (0.2,3);

\draw[purple] (-0.2, -3) to[in=180,out=90] (0,-2.7);
\draw[purple] (0, -2.7) to[in=90,out=0] (0.2,-3);

\node[purple] at ([shift=({360/10*(7)}:\radius+1.5cm)]0.5,0.4) {\small $T_2$}; 
\node[purple] at ([shift=({360/10*(2)}:\radius+1.5cm)]-0.5,-0.4) {\small $T_1$};

\node[vertex] at  ([shift=({360/10*(7.5)}:\radius+1.5cm)]0,0) {};
\node[vertex] at  ([shift=({360/10*(2.5)}:\radius+1.5cm)]0,0) {};

\node at (-1.5,0.55){\tiny $u_1$};
\node at (-1.5,-0.6){\tiny $v_1$};

\node at (0,0.55){\tiny $u_2$};
\node at (0,-0.6){\tiny $v_2$};

\node at (1.5,0.55){\tiny $u_3$};
\node at (1.5,-0.6){\tiny $v_3$};

\end{tikzpicture}
    \caption{Sphere with 4 holes, 2 stops, a generating set of Lagrangians ($L_i$) and certain other Lagrangians supported near stops 
          ($T_1$ and $T_2$)}
    \label{puncsphere}
\end{figure}
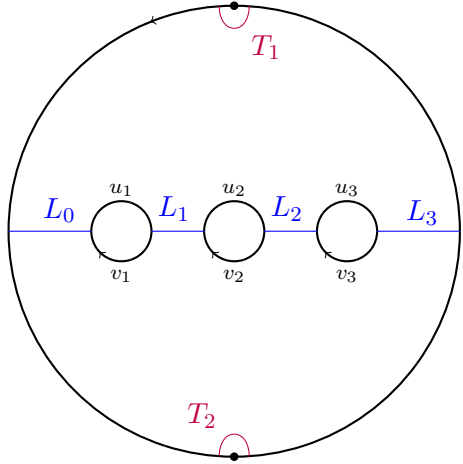

At the $i^{th}$ interior puncture, we write $u_i,v_i$ for the two primitive Reeb chords 
\[ u_i \in hom_{\mathcal{W}(\Sigma_k, Z)} (L_{i-1}, L_{i}), \ \  v_i \in hom_{\mathcal{W}(\Sigma_k, Z)} (L_{i}, L_{i-1}), \]
as in Figure \ref{puncsphere}. 

\subsection{Case of 2-dimensional pairs-of-pants}

As a warm-up, let us consider the special case $n=2, k=3$. The symplectic manifold $M_{2,3}$ is also known as the 2-dimensional pair-of-pants. The category $\mathcal{W}(M_{2,3}, \Lambda)$ is generated by ${4 \choose 2}= 6$ Lagrangians, and the following proposition computes all the morphisms between them.

\begin{prop} 
\label{n2k3}
We have natural identifications
\begin{align*}
\End(L_{2} \times L_{3}) &= R/ (x_1)  \\
\End(L_{0} \times L_{3}) &=  R/ (x_2)  \\
\End(L_{0} \times L_{1}) &= R/ (x_3)  \\
\End(L_{1} \times L_{3}) &= R/ (x_1x_2) \\ 
\End(L_{0} \times L_{2}) &= R/ (x_2x_3) \\
\End(L_{1} \times L_{2}) &= R/ (x_1x_2x_3)  \\ 
\end{align*}
The morphisms between these objects are encoded by the following quiver over $R$ with relations:
\begin{figure}[!h]
\begin{tikzpicture}
    \tikzset{vertex/.style = {style=circle,draw, fill,  minimum size = 2pt,inner    sep=1pt}}

        \tikzset{edge/.style = {->,-stealth',shorten >=8pt, shorten <=8pt  }}

\tikzset{
  with arrows/.style={->, -stealth}, with arrows/.default=2mm,
}

\node[vertex] (a) at  (0,0) {};
\node[vertex] (a1) at (2.5,2) {};
\node[vertex] (a2) at (2.5,-2) {};
\node[vertex] (a3) at (5,0) {};
\node[vertex] (a4) at (-2.5,-2) {};
\node[vertex] (a5) at (7.5,-2) {};

\node at (-0.5,0.3){\tiny $L_0 \times L_2$};
\node at (2.5,2.3) {\tiny $L_1 \times L_2$};
\node at (2.5,-2.3){\tiny $L_0 \times L_3$};
\node at (5.5,0.3) {\tiny $L_1 \times L_3$};
\node at (-3, -1.7) {\tiny $L_0 \times L_1$};
\node at (8,-1.7){\tiny $L_2 \times L_3$};

\node at (-0.9, -1.5) {\tiny $v_2$};
\node at (-1.5, -0.4) {\tiny $u_2$};

\node at (6.5, -0.4) {\tiny $u_2$};
\node at (6, -1.5) {\tiny $v_2$};

\node at (0.8, 1.5) {\tiny $u_1$};
\node at (1.7, 0.7) {\tiny $v_1$};
\node at (0.8, -1.5) {\tiny $v_3$};
\node at (2.1, -0.7) {\tiny $u_3$};
\node at (4.2, -1.5) {\tiny $v_1$};
\node at (3.2, -0.7) {\tiny $u_1$};
\node at (3.2, 0.7) {\tiny $v_3$};
\node at (4.2, 1.5) {\tiny $u_3$};


\draw[edge] (a)  to[out=55,in=190]  (a1);
\draw[edge] (a1)  to[out=230,in=20]  (a);
\draw[edge] (a)  to[out=-10,in=120] (a2);
\draw[edge] (a2)  to[out=160, in=310] (a);
\draw[edge] (a1) to[out=-10, in=120] (a3);
\draw[edge] (a3) to[out=160, in=310] (a1);

\draw[edge] (a2) to[out=55,in=190] (a3);
\draw[edge] (a3) to[out=230,in=20] (a2);

\draw[edge] (a4) to[out=55, in=190] (a);
\draw[edge] (a) to[out=230, in=20] (a4);

\draw[edge] (a3) to[out=-10, in=120] (a5);
\draw[edge] (a5) to[out=160, in=310] (a3);







\end{tikzpicture}
	\caption{Morphisms between a generating set of objects.}
\label{figure2}
\end{figure}
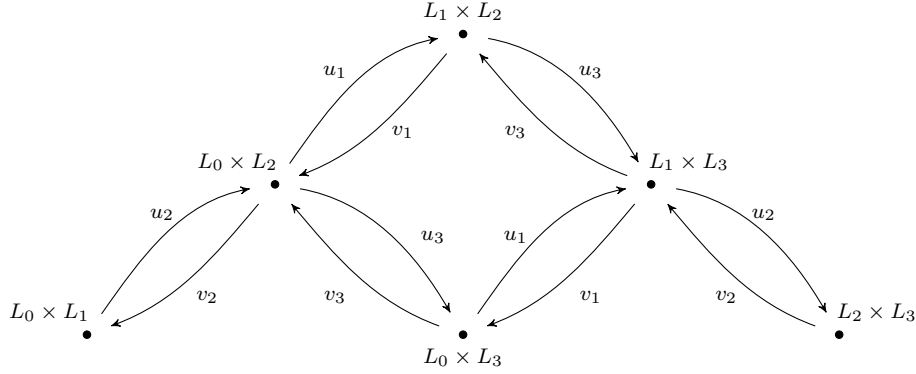

\noindent
where  
\[ u_i v_i = x_i = v_i u_i, \ \  u_3 u_2 = v_2v_3 = u_2 u_1 = v_1 v_2 =0\]
and 
\[ u_3 u_1 = u_1 u_3, v_3 v_1 = v_1 v_3, u_3 v_1 = v_1 u_3, u_1 v_3 = v_3 u_1\]

\end{prop} 
\begin{proof} When $K$ and $L$ do not have endpoints on the same boundary component of $\Sigma$, the generators for $\Endlc(K \times L)$ are just given by $ k \otimes \id, \id \otimes l$ where $k \in \End(K)$ and $l \in \End(L)$. 
The differential on $\Endlc(K\times L)$ is zero, and $k$ and $l$ commute. For example, 
\[ \End(L_{0} \times L_{3}) = \k \langle \id \otimes x_3, x_1 \otimes \id \rangle =  R/ (x_2). \]
On the other hand, when both $K$ and $L$ have an end in the $i$th boundary component of $\Sigma$, 
we have that $end(K \times L)$ contains the vector subspace spanned by the elements
\begin{equation}\label{end-KL-subspace-eq} 
(v_iu_i)^m \otimes (u_iv_i)^n, u_i(v_iu_i)^m \otimes v_i(u_iv_i)^n, \text{ for } m,n \geq 0. 
\end{equation}
To understand the algebra structure let us set 
$$a_i=v_iu_i \otimes \id_L, \ \ b_i=\id_K \otimes u_iv_i, \ \ c_i=u_i\otimes v_i.$$
Then we have the relations
$$a_ib_i=b_ia_i=0, \ \ a_ic_i=c_ib_i, \ \ b_ic_i=c_ia_i,$$
where the first relation comes from the product rule explained in Figure \ref{notallow}. 

The quadratic algebra with these relations has the Gr\"obner bases
$$(c_i^n, \ a_i^mc_i^n, \ b_i^mc_i^n)_{n\ge 0, m>0},$$
which is exactly the elements \eqref{end-KL-subspace-eq}.
The differential is given by 
\[ \partial (a_i) = -\partial(b_i)= c_i, \ \partial(c_i)=0 \]
and extended by the graded Leibniz rule, where we have $\deg(a_i)=\deg(b_i)=0$, $\deg(c_i)=1$.
It is easy to check that the relations are preserved. This also determines the signs. Furthermore, we have
$$\partial(c_i^n)=0, \ \ \partial(a_ic_i^n)=-\partial(b_ic_i^n)=c_i^{n+1}, \ \ 
\partial(a_i^mc_i^n)=-\partial(b_i^mc_i^n)=(a_i^{m-1}+b_i^{m-1})c_i^{n+1},$$
where $m\ge 2$.

If $K$ and $L$ end at the $i^{th}$ boundary component of $\Sigma$, we let 
$$x_i=a_i+b_i=v_i u_i \otimes \id_{L} + \id_{K} \otimes u_i v_i.$$ 

One can see by a straightforward calculation from the explicit description of the chain complex given above, that the contribution to the cohomology $\End(K \times L)$ from the $i^{th}$ boundary components comes from $x_i$ and its positive powers. This determines the cohomology. For example,
\[ \End(L_{1} \times L_{2}) = \k \langle u_1v_1 \otimes \id, v_2u_2 \otimes \id + \id \otimes u_2 v_2, \id \otimes v_3 u_3 \rangle = R/ (x_1x_2x_3) \] 

The morphisms between different Lagrangians are calculated in the same way 
(see Theorem \ref{A-side-morphisms-general-thm} below for a more general calculation). 
\end{proof}

Let us record one simple computation used above.

\begin{lem}\label{uv-acyclic-lem} 
Let us consider the subcomplex of $\Endlc(K\times L)$ spanned by the elements
$$((v_iu_i)^m\otimes (u_iv_i)^n)_{m\ge 0, n>0}, \ \ (u_i(v_iu_i)^m\otimes v_i(u_iv_i)^n)_{m\ge 0,n\ge 0}.$$
Then this subcomplex is exact.
\end{lem}

\Pf . In terms of the generators $a_i,b_i,c_i$, our subcomplex is spanned by the elements
$$(a_i^mc_i^n)_{m\ge 0,n>0}, \ \  (b_i^mc_i^n)_{m>0,n\ge 0}.$$
This complex splits into a direct sum of subcomplexes with fixed total degree (given by $m+n$).
The subcomplex $C^{\bullet}$ corresponding to the degree $n>0$ has terms
\begin{align*}
&C^0=\lan b_i^n\ran, \ \ C^1=\lan a_i^{n-1}c_i, b_i^{n-1}c_i\ran, \ \ C^2=\lan a_i^{n-2}c_i^2, b_i^{n-2}c_i^2\ran,\ldots,\\
&C^{n-1}=\lan a_ic_i^{n-1}, b_ic_i^{n-1}\ran, \ \ C^n=\lan c_i^n\ran.
\end{align*}
Now we see that for $m\in [1,n-1]$, one has
$$\ker(d:C^m\to C^{m+1})=\lan (a_i^{n-m}+b_i^{n-m})c_i^m\ran=\im(d:C^{m-1}\to C^m),$$
while $\ker(d:C^0\to C^1)=0$ and $\im(d:C^{n-1}\to C^n)=C^n$.
\ed

\subsection{General $n,k$ with $k \geq n$}

We now describe the computation for arbitrary $k,n$ with $k \geq n$. Let $L_{0}, \ldots, L_{k-1}, L_{k}$ be the arcs that generate $\mathcal{W}(\Sigma, Z)$ as before. The generators of $\mathcal{W}(M_{n,k}, \Lambda)$ are given by the ${k+1 \choose n}$ Lagrangians. Let 
\[ L_S = L_{i_1 } \times L_{i_2} \times \ldots \times L_{i_n} , \text{ with } S=\{i_1<i_2<\ldots<i_n\}\sub [0,k].\]

In order to understand morphisms between $L_S$ and $L_{S'}$ we need the following combinatorial statement.

\begin{prop}\label{perm-prop}
Let $S,S'\sub [0,k]$ be a pair of size $n$ subsets, and let $g:S\to S'$ be a bijection such that for every 
$i\in S$ one has $g(i)\in \{i-1,i,i+1\}$. Then there exists a collection of disjoint subintervals $I_1,\ldots,I_r\sub [0,k]$
such that $S\setminus \sqcup_j I_j=S'\setminus \sqcup_j I_j$; $g(i)=i$ for $i\in S\setminus \sqcup_j I_j$;
and for each subinterval $I_j$ one of the following holds:
\begin{enumerate}
\item $I_j=[i,i+1]\sub S\cap S'$ and $g$ swaps $i$ with $i+1$;
\item $I_j=[a,b]$, $S\cap I_j=[a,b-1]$, $S'\cap I_j=[a+1,b]$, and $g(i)=i+1$ for $i\in S\cap I_j$;
\item $I_j=[a,b]$, $S'\cap I_j=[a,b-1]$, $S\cap I_j=[a+1,b]$, and $g(i)=i-1$ for $i\in S\cap I_j$.
\end{enumerate}
\end{prop}

We need some preparations before giving a proof. For a pair of size $n$ subsets $S,S'\sub [0,k]$, let us set
$$T=T(S,S'):=\{ i \ |\ \#(S\cap [0,i])=\#(S'\cap [0,i])\}.$$
We can write $T$ as the union of disjoint intervals,
$$T=[s_1,t_1]\sqcup [s_2,t_2]\sqcup\ldots [s_r,t_r],$$
where $t_r=k$, $s_i\le t_i$ and $t_i+1<s_{i+1}$.

\begin{lem}\label{perm-lem} 
(i) Let $g:S\to S'$ be a bijection such that for every $i\in S$ one has $g(i)\in \{i-1,i,i+1\}$.
Then $g$ induces bijections
$$S\cap [s_i+1,t_i]\rTo{\sim} S'\cap [s_i+1,t_i], 2\le i\le r, \ \ 
S\cap [t_i+1,s_{i+1}]\rTo{\sim} S'\cap [t_i+1,s_{i+1}], 1\le i\le r-1.$$
In addition, if $s_1=0$ then
$$g(S\cap [0,t_1])=S\cap [0,t_1]=S'\cap [0,t_1],$$
and if $s_1>0$ then 
$$g(S\cap [0,s_1])=S'\cap [0,s_1], \ \ g(S\cap [s_1+1,t_1])=S'\cap [s_1+1,t_1].$$
Furthermore, we have
$$S\cap [s_i+1,t_i]=S'\cap [s_i+1,t_i]$$
(note that these intervals could be empty). On the other hand,
each restriction
$$g: S\cap [t_i+1,s_{i+1}]\to S'\cap [t_i+1,s_{i+1}]$$
(resp., $g:S\cap [0,s_1]\to S'\cap [0,s_1]$ if $s_1>0$)
is given either by $g(j)=j-1$ or $g(j)=j+1$.

\noindent
(ii) Let $g:S\to S$ be a permutation such that for every $i\in S$ one has $g(i)\in \{i-1,i,i+1\}$.
Then there exists a subset $S_0\sub S$ of the form
$$S_0=\{i_1,i_1+1,i_2,i_2+1,\ldots,i_r,i_r+1\},$$
where $i_s+1<i_{s+1}$, such that $g$ swaps $i_s$ and $i_{s+1}$ for $s=1,\ldots,r$, and $g(i)=i$ for every $i\in S\setminus S_0$.
\end{lem}

\Pf . (i) Note that for every $i$ we have $g(S\cap [0,i])\sub S'\cap [0,i+1]$ and $g^{-1}(S'\cap [0,i])\sub S\cap [0,i+1]$.
Now, since $t_i\in T$ and $t_i+1\not\in T$, we have either $t_i+1\in S\setminus S'$ or $t_i+1\in S'\setminus S$.
In the former case we have
$$g(S\cap [0,t_i])\sub S'\cap [0,t_i+1]=S'\cap [0,t_i],$$
while in the latter case we have
$$g^{-1}(S'\cap [0,t_i])\sub S\cap [0,t_i+1]=S\cap [0,t_i].$$
Thus, we get $g(S\cap [0,t_i])=S'\cap [0,t_i]$.

Next, we have $s_i-1\not\in T$ and $s_i\in T$, so either $s_i\in S\setminus S'$ or $s_i\in S'\setminus S$. In the former case
we have
$$g(S\cap [s_i+1,k])\sub S'\cap [s_i,k]=S'\cap [s_i+1,k],$$
while in the latter case we have
$$g^{-1}(S'\cap [s_i+1,k])\sub S\cap [s_i,k]=S\cap [s_i+1,k].$$
Since $\#(S\cap [s_i+1,k])=\#(S'\cap [s_i+1,k])$, we deduce that $g(S\cap [s_i+1,k])=S'\cap [s_i+1,k]$.
This implies our first assertion.

Assume that $s_i+1\le t_i$. Since both $s_i$ and $s_i+1$ are in $T$, either the element 
$s_i+1$ belongs to $S\cap S'$, or it does not belong to both $S$ and $S'$. Proceeding in the same way we see
that $S\cap [s_i+1,t_i]=S'\cap [s_i+1,t_i]$.

Next, let us consider the interval $[t_i+1,s_{i+1}]$. For every $i$ in this interval, except for the right end,
we have $i\not\in T$. Assume first that $t_i+1\in S$. Then, since $t_i\in T$ and $t_i+1\not\in T$, 
we should have $t_i+1\not\in S'$, so $g(t_i+1)=t_i+2\in S'$. If $t_i+2<s_{i+1}$ we can continue in the same way:
since $\#(S\cap[0,t_i+1])=\#(S'\cap[0,t_i+1])+1$ and $t_i+2\in S'\setminus T$, we should have $t_i+2\in S$, and so
$g(t_i+2)=t_i+3$, etc. In the case $t_i+1\not\in S$, we should have $t_i+1\in S'$, so we can apply the same argument
with $S$ and $S'$ swapped and $g$ replaced by $g^{-1}$.

\noindent
(ii) We use induction on the cardinality of $S$.
Let $i_0$ be the minimal element of $S$. Then we have either $g(i_0)=i_0$ or $g(i_0)=i_0+1$. In the former
case we have $g(S\setminus \{i_0\})=S\setminus \{i_0\})$, so we can apply the induction assumption to
$S\setminus\{i_0\}$. Now let us assume that $g(i_0)=i_0+1$. Then $g^{-1}(i_0)\neq i_0$, so we should have
$g^{-1}(i_0)=i_0+1$. Thus, in this case $g$ swaps $i_0$ and $i_0+1$. Now we can replace $S$ by 
$S\setminus \{i_0,i_0+1\}$ and apply the induction assumption.
\ed

\noindent
{\it Proof of Proposition \ref{perm-prop}}. By Lemma \ref{perm-lem}(i), we can partition $[0,k]$ into subintervals of two kinds:
those for which one of the conditions (2) or (3) of Proposition \ref{perm-prop} holds, and subintervals $I$ such that
$S\cap I=S'\cap I$ and $g(S\cap I)=S'\cap I$. It remains to apply Lemma \ref{perm-lem}(ii) to all subintervals of the second kind.
\ed

\begin{definition}
We say that two size $n$ subsets $S,S'\sub [0,k]$ are {\it close} if
there exists a bijection $g:S\to S'$ such that $g(i)\in \{i-1,i, i+1\}$ for every $i$.
\end{definition}

\begin{cor}\label{pair-partition-cor} 
If size $n$ subsets $S,S'\sub [0,k]$ are close
then there exists a unique collection of disjoint subintervals of $[0,k]$,
$(I_1,\ldots,I_r$, $J_1,\ldots,J_s)$,
such that 
\begin{itemize}
\item $([0,k]\setminus \sqcup_a I_a\sqcup\sqcup_b J_b)\cap S=([0,k]\setminus \sqcup_a I_a\sqcup\sqcup_b J_b)\cap S'$;
\item for each $a=1,\ldots,r$, if $I_a=[i,j]$ then $I_a\cap S=[i,j-1]$, $I_a\cap S'=[i+1,j]$;
\item for each $b=1,\ldots s$, if $J_b=[i,j]$ then $J_b\cap S'=[i,j-1]$, $J_b\cap S=[i+1,j]$.
\end{itemize} 
\end{cor}

For every proper subinterval $[i,j]\sub[0,k]$, let us set 
$$\cA_{[i,j]}=\begin{cases} \k [x_i,\ldots,x_{j+1}]/(x_i\ldots x_{j+1}) & \text{ if } i>0, j<k, \\ 
\k[x_1,\ldots,x_{j+1}] & \text { if } i=0, j<k, \\ 
\k[x_i,\ldots,x_k] & \text { if } i>0, j=k,\end{cases}$$ 
$$\cA'_{[i,j]}=\k [x_{i+1},\ldots,x_j].$$



Finally, for each interval $[i,j]$ we consider the elements
$$u_{[i,j]}:=u_{i+1}\ot u_{i+2}\ot\ldots \ot u_j\in \homs(L_i,L_{i+1})\ot\homs(L_{i+1},L_{i+2})\ot\ldots\homs(L_{j-1},L_j),$$
$$v_{[i,j]}:=v_{i+1}\ot v_{i+2}\ot\ldots \ot v_j\in \homs(L_{i+1},L_i)\ot\homs(L_{i+2},L_{i+1})\ot\ldots\homs(L_j,L_{j-1}).$$

\begin{thm}\label{A-side-morphisms-general-thm}
\noindent
(i) For every $k$-tuple of integers $(d_1,\ldots,d_k) \in \mathbb{Z}^k$ there exists a unique grading structure on $M_{n,k}$ such that for $S=[i_1,j_1]\sqcup [i_2,j_2]\sqcup \ldots\sqcup [i_r,j_r]$ with $j_s+1<i_{s+1}$, one has a natural isomorphism
of graded algebras
$$\End(L_S)\simeq \cA(S,S):=\cA_{[i_1,j_1]}\ot \cA_{[i_2,j_2]}\ot\ldots \ot\cA_{[i_r,j_r]},$$
where we have $\deg(x_i)=d_i$.

\noindent
(ii) For a pair of size $n$ subsets $S, S'\sub [0,k]$, one has $\Hom(L_S,L_{S'})=0$ if $S$ and $S'$ are not close.
If $S$ and $S'$ are close then there is a natural identification
\begin{equation}\label{hom-S-S'-isom}
\Hom(L_S,L_{S'})\simeq \cA(S,S')\cdot f_{S,S'}, \text{ with } \cA(S,S'):=\cA(S_0,S_0)\ot
\bigotimes_{a=1}^r \cA'_{I_a}\ot\bigotimes_{b=1}^s \cA'_{J_b},
\end{equation}
where we use the subintervals $(I_a)$, $(J_b)$ from Corollary \ref{pair-partition-cor}, and set
$$S_0=S\setminus (\sqcup_a I_a\sqcup\sqcup_b J_b)=S'\setminus (\sqcup_a I_a\sqcup\sqcup_b J_b),$$
$$f_{S,S'}:=\bigotimes_{a=1}^r u_{I_a}\ot\bigotimes_{b=1}^s v_{J_b}\ot\id\in \Hom(L_S,L_{S'})$$
Furthermore, \eqref{hom-S-S'-isom} is an isomorphism of $\End(L_{S'})-\End(L_S)$-bimodules, where the structure of 
a $\cA(S',S')-\cA(S,S)$-bimodule on $\cA(S,S')$ is induced by the surjective homomorphisms of $\k$-algebras
$$\cA(S,S)\to \cA(S,S'), \ \ \cA(S',S')\to \cA(S,S')$$
sending each $x_i$ either to $x_i$ or to $0$ (if $x_i$ is absent in $\cA(S,S')$).

\noindent
(iii) The compositions are uniquely determined by the bimodule structures on $\Hom(L_S,L_{S'})$ together with
the rule
\begin{equation}\label{generators-composition-eq}
f_{S',S''}f_{S,S'}=\begin{cases} (\prod_{[i-1,i]\sub T(S,S',S'')} x_i)\cdot f_{S,S''}, & S \text{ and } S'' \text{ are close},\\
0, & \text{otherwise}.\end{cases}
\end{equation}
Here $(I_a)$, $(J_b)$ (resp., $(I'_{a'})$, $(J'_{b'})$) are the subintervals of Corollary \ref{pair-partition-cor} for the pair $S,S'$
(resp., $S',S''$), and
$$T(S,S',S'')=\bigl((\sqcup_a I_a)\cap(\sqcup_{b'} J'_{b'})\bigr)\cup 
\bigl((\sqcup_{a'} I'_{a'})\cap(\sqcup_b J_b)\bigr).$$
which can be empty in which case we put $f_{S',S''}f_{S,S'}= f_{S,S''}$.

\noindent
(iv) The dg-algebra $\bigoplus_{S,S'} \mathrm{hom}(L_S,L_{S'})$ contains a quasi-isomorphic dg-subalgebra with the trivial
differential, in particular, it is formal.

          For any abelian group $D$, and any given assignment of degrees, $\deg(x_i)=d_i\in D$, $i=1,\ldots,k$,
there is at most one $D$-grading on the algebra 
$$\cA^{\circ\circ}=\bigoplus_{S,S'}\Hom(L_S,L_{S'})$$
coming from some choices of $\deg(f_{S,S'})=d_{S,S'}\in D$, up to a transformation of the form
$d_{S,S'}\mapsto d_{S,S'}+d_{S'}-d_S$. For $D=\Z$ and any $d_i\in\Z$, such a $\Z$-grading exists.
\end{thm}

\begin{proof} (i) Recall that a grading structure on $M_{n,k}$ is given by a fiberwise universal cover of the Lagrangian Grassmannian of the tangent bundle (\cite{seidelgraded}). Such a cover exists since $2c_1(M_{n,k})=0$, and the set of
possible gradings structures up to homotopy is a torsor over $H^1(M_{n,k},\Z) \simeq \mathbb{Z}^k$. As all our Lagrangians $L_S$ are contractible, they can be graded (uniquely up to a shift by $\Z$).
Changing the homotopy class of the grading by some cohomology class $c \in H^1(M_{n,k};\Z)\simeq H^1(M_{n,k},L;\Z)$, 
changes the degree of a Reeb chord $x \in \Endlc(L_S)$ by $\langle c , [x] \rangle$ where 
$[x] \in H_1(M_{n,k},L_S;\Z)$ (see \cite[Sec. 9a]{abouzseidel}). 

Now we use the fact that away from the big diagonal $\De\sub M_{n,k}$ the Reeb flow is just the product of the Reeb flows
on the surface. In addition, we observe that the homotopy class of any Reeb loop of the form 
$(pt. \times \ldots \times pt. \times \ga_i\times pt. \times \ldots \times pt.)$ in $M_{n,k}\setminus\De$, 
coming from the loop $\ga_i=u_iv_i$ in $\Sigma$, for fixed $i$ does not depend on the choice of points in the other components. It follows that the degree of the generator of $\Endlc(L_S)$ corresponding to such a loop is independent of $S$ (it depends
only on $i$).
Hence, the fact that grading structures on $M_{n,k}$ form a torsor over $H^1(M_{n,k},\Z)$ implies that for
a given $(d_1,d_2,\ldots, d_k) \in \mathbb{Z}^k$, there exists a unique grading structure such that $|x_i|=d_i$. 

Let us first calculate $\End(L_S)$ in the case $S=[i,j]$, where $i>0$, $j<k$. Let us set
$$E[i,j]:=\Endlc(L_{[i,j]}).$$
First, we claim that there is a natural quasiisomorphism of dg-algebras
\begin{equation}\label{end-H0-map}
k[x_i,\ldots,x_{j+1}]/(x_i\ldots x_{j+1})\to E[i,j],
\end{equation}
such that 
\begin{align*}
&x_i\mapsto (u_iv_i)\otimes \id \otimes \ldots, \\ 
&x_s\mapsto 
\ldots\otimes \id \otimes(v_su_s)\otimes\id\otimes\ldots+\ldots \otimes \id\otimes (u_sv_s)\otimes \id \otimes \ldots
\ \text{ for } i<s\le j,\\
&x_{j+1}\mapsto\ldots \otimes \id \otimes (v_{j+1}u_{j+1}). 
\end{align*}
Here we view the source as a complex with zero
differential.
It is easy to check that the morphisms corresponding to $x_s$ are closed, pairwise commute and their product is zero,
so the map \eqref{end-H0-map} is well defined.

Next, we want to prove that \eqref{end-H0-map} is a quasi-isomorphism.
We proceed by induction on $j-i$. In the base case $j=i$ this is easy to see.
For the induction step we start with a decomposition into a direct sum of subcomplexes,
$$E[i,j]=\bigoplus_{n\ge 0} E[i,j](n).$$
Here for $n>0$, $E[i,j](n)$ is spanned by elements of the form $(u_iv_i)^n\otimes \ldots$, while
$E[i,j](0)$  is spanned by the remaining basis elements.

It is easy to see that for $n>0$, $E[i,j](n)$ is isomorphic (by multiplication with $x_i^n$) with 
$E[i+1,j]$, so we know its cohomology from the induction assumption. More precisely,  we know
that the map 
$$x_i^n\cdot k[x_{i+1},\ldots,x_{j+1}]/(x_{i+1}\ldots x_{j+1})\to E[i,j](n)$$
is a quasiisomorphism.

It remains to deal with $E[i,j](0)$. Namely, we want to prove that 
the natural map 
$$k[x_{i+1},\ldots,x_{j+1}]\to E[i,j](0)$$
is a quasiisomorphism.

Let us define subcomplexes $C(0)\sub C(1)\sub E[i,j](0)$ as follows.
We define $C(0)$ to be the span of all elements of the form
\begin{align*}
&((v_{i+1}u_{i+1})^m\otimes (u_{i+1}v_{i+1})^n\otimes\id_{L_{i+2}}\otimes\ldots)_{m\ge 0,n>0}, \\ 
&(u_{i+1}(v_{i+1}u_{i+1})^m \otimes v_{i+1}(u_{i+1}v_{i+1})^n\otimes\id_{L_{i+2}}\otimes\ldots)_{m\ge 0, n\ge 0}.
\end{align*}
By Lemma \ref{uv-acyclic-lem}, the complex $C(0)$ is acyclic.

Next, we define $C(1)$ to be spanned by $C(0)$ and by all elements of the form $id_{L_i}\otimes\ldots$
We have an isomorphism 
$$E[i+1,j](0)\rTo{\sim} C(1)/C(0)$$ 
induced by the natural embedding of $E[i+1,j](0)$ into $C(1)$.
Thus, by the induction assumption,
the natural map
$$k[x_{i+2},\ldots,x_{j+1}]\to C(1)/C(0)$$
is a quasiisomorphism.

Finally, the quotient $E[i,j](0)/C(1)$ splits into a direct sum of subcomplexes $E(n)$ numbered by $n>0$,
where $E(n)$ is spanned by elements of the form $(v_{i+1}u_{i+1})^n\otimes\ldots$ (modulo $C(1)$).
Note that we have
$$x_{i+1}^n\equiv (v_{i+1}u_{i+1})^n\otimes\id\otimes\ldots \mod C(1),$$
and it is easy to see that the multiplication by $x_{i+1}^n$ gives an isomorphism of complexes,
$$E[i+1,j](0)\rTo{\sim} E(n).$$ 
Therefore, we deduce that the natural map
$$x_{i+1}^n k[x_{i+2},\ldots,x_{j+1}]\to E(n)$$
is a quasiisomorphism for every $n>0$.

Thus, we have a morphism of exact triples of complexes
\begin{diagram}
0& \rTo{}& k[x_{i+2},\ldots,x_{j+1}] &\rTo{}& k[x_{i+1},\ldots,x_{j+1}]&\rTo{}&\bigoplus_{n>0}x_{i+1}^nk[x_{i+2},\ldots,x_{j+1}]&\rTo{}
& 0\\
&&\dTo{}&&\dTo{}&&\dTo{}\\
0&\rTo{} & C(1)/C(0)&\rTo{}& E[i,j](0)/C(0)&\rTo{}& E[i,j](0)/C(1)&\rTo{}& 0
\end{diagram}
in which the left and right vertical arrows are quasiisomorphisms, hence, the middle one is also a quasiisomorphism.

Next, we can slightly modify the above argument for the cases $i=0$ and $j<k$ (resp., $i>0$ and $j=k$).
In these cases we still claim that the 
natural maps
$$k[x_1,\ldots,x_{j+1}]\to E[0,j], \ \ k[x_i,\ldots,x_k]\to E[i,k]$$
are quasiisomorphisms. Namely, in the case of $E[0,j]$ we skip the first step and just deal with this complex exactly as with
the complexes $E[i,j](0)$, considering the subcomplexes analogous to $C(0)$ and $C(1)$.
In the case of $E[i,k]$, we repeat exactly the same steps as for $E[i,j]$, considering first the subcomplexes
$E[i,k](n)$ and using the induction assumption to see that 
the maps
$$x_i^n\cdot k[x_{i+1},\ldots,x_k]\to E[i,k](n)$$
are quasiisomorphisms.

Now for $S=[i_1,j_1]\sqcup [i_2,j_2]\sqcup \ldots\sqcup [i_r,j_r]$ with $j_s+1<i_{s+1}$, we have a natural identification
$$\Endlc(L_S)\simeq \Endlc(L_{[i_1,j_1]})\ot \ldots\ot \Endlc(L_{[i_r,j_r]}),$$
and so we have a similar decomposition of $\End(L_S)$.

\noindent
(ii) Now let us compute $\Hom(L_S,L_{S'})$ for a pair of size $n$ subsets $S,S'\sub [0,k]$.
If $S$ and $S'$ are not close then we have $\homs(L_S,L_{S'})=0$.
Otherwise, we have a decomposition
\begin{equation}\label{hom-intervals-decomposition}
\homs(L_S,L_{S'})\simeq \Endlc\bigl(L_{([0,k]\setminus \sqcup_a I_a\sqcup\sqcup_b J_b)\cap S}\bigr)\ot 
\bigotimes_{a=1}^r \homs(L_{I_a\cap S},L_{I_a\cap S'})\ot\bigotimes_{b=1}^s \homs(L_{J_b\cap S},L_{J_b\cap S'}),
\end{equation}
where we use the subintervals $(I_a)$, $(J_b)$ from Corollary \ref{pair-partition-cor}.

It remains to identify the complexes $\homs(L_{I_a\cap S},L_{I_a\cap S'})$
and $\homs(L_{J_b\cap S},L_{J_b\cap S'})$. First, let $I=I_a=[i,j]$, so that $I\cap S=[i,j-1]$ and $I\cap S'=[i+1,j]$.
We have
$$\homs(L_{[i,j-1]},L_{[i+1,j]})=\bigotimes_{p\in [i,j-1]}\homs(L_p,L_{p+1}),$$
with the basis given by
$$u_{i+1}(v_{i+1}u_{i+1})^{m_1}\ot\ldots\ot u_j(v_ju_j)^{m_{j-i}}, \ \ m_1\ge 0,\ldots,m_{j-i}\ge 0.$$
It is easy to check that we have
$$u_{i+1}(v_{i+1}u_{i+1})^{m_1}\ot\ldots\ot u_j(v_ju_j)^{m_{j-i}}=
(x_{i+1}^{m_1}\ldots x_j^{m_{j-i}}) u_{[i,j]}= u_{[i,j]} (x_{i+1}^{m_1}\ldots x_j^{m_{j-i}}),$$
where we view the monomial in $x_1,\ldots,x_k$ as an endomorphism of $L_{[i,j-1]}$ or of $L_{[i+1,j]}$.
In particular, the differential on $\homs(L_{[i,j-1]},L_{[i+1,j]})$ is zero.

Similarly, we see that for $J=J_b=[i,j]$, the space $\homs(L_{[i+1,j]},L_{[i,j-1]})$ is spanned by
$$(x_{i+1}^{m_1}\ldots x_j^{m_{j-i}}) v_{[i,j]}= v_{[i,j]} (x_{i+1}^{m_1}\ldots x_j^{m_{j-i}}),$$
and has zero differential.

\noindent
(iii) Formula \eqref{generators-composition-eq} can be checked directly. To compose arbitrary elements
in $\Hom(L_S,L_{S'})$ and $\Hom(L_{S'},L_{S''})$ we can use (ii) to write them in the form
$f_{S,S'}p(x)$ and $q(x)f_{S',S''}$, with $p(x)$ and $q(x)$ some polynomials in $x_1,\ldots,x_k$, and then use
\eqref{generators-composition-eq}.

\noindent
(iv) The first assertion is clear from our previous computations: the subalgebra in question is generated by $f_{S,S'}$ and by
$x_i\in \mathrm{end}(L_S)$.
          
          By (i), we know that for each choice $\deg(x_i)=d_i\in\Z$,
there exists a $\Z$-grading on $\cA^{\circ\circ}$ coming from a grading structure on $M_{n,k}$. Thus, it is enough
to prove uniqueness of $(d_{S,S'})$ up to adding $d_{S'}-d_S$ (for given $d_i\in D$).
The numbers $d_{S,S'}$ are constrained by equations \eqref{generators-composition-eq}, corresponding
to triples $(S,S',S'')$ of pairwise close subsets. Thus, a difference between two systems $(d_{S,S'})$ is a $1$-cocycle
for the following simplicial complex $X$. The set of vertices, $X_0$, is the set of all size $n$ subsets in $[0,k]$.
The set of edges, $X_1$, is given by pairs $(S,S')$ of subsets which are close, and the set of $2$-simplices, $X_2$,
is the set of triples $(S,S',S'')$ of pairwise close subsets. Our uniqueness statement would follow from the vanishing
of $H^1(X,D)$. We claim that in fact the geometric realization $|X|_{\real}$ is simply connected. 
The proof is by induction on $k$. The case $k=1$
is straightforward (for $n=0$ we get a point and for $n=1$ we get a segment).
Another special case we have to include to make our proof work is $n=k+1$, when $X$ reduces to a point.

Let us write $X=X(n,k)$. We have natural embeddings  as simplicial subcomplexes,
$$i_k:X(n,k-1)\hra X(n,k): S\sub [0,k-1]\mapsto S,$$
$$j_k:X(n-1,k-1)\hra X(n,k): S\sub [0,k-1]\mapsto S\sqcup\{k\},$$
so that every vertex of $X(n,k)$ is either in the image of $i_k$ or in the image of $j_k$.

Let us in addition consider the subcomplex $Y\sub X(n,k)$ spanned by all vertices of the form
$S\sqcup\{k-1\}$ and $S\sqcup\{k\}$, with $S\sub [0,k-2]$, $|S|=n-1$. Note that whenever $S$ and $S'$ are close
subsets of $[0,k-2]$, all four subsets
$$S\sqcup\{k-1\}, \ \ S\sqcup\{k\}, \ \ S'\sqcup\{k-1\}, \ \ S'\sqcup\{k\}$$
are pairwise close. This implies that $|Y|_{\real}$ can be identified with the $2$-skeleton of
$|X(n-1,k-2)|_{\real}\times I$, where $I$ is a segment. In particular, by the induction assumption, $|Y|_{\real}$
is simply connected.

Now assume that we have an edge $e$ between a vertex $S$ in $i_k(X(n,k-1))$, where $S\subset [0,k-1]$
and a vertex $S'\sqcup\{k\}$ in $j_k(X(n-1,k-1))$. This means that $S$ and $S'\sqcup\{k\}$ are close,
which can happen only if $S=T\sqcup\{k-1\}$, with $T$ and $S'$ close. In this case the triangle
$$(S=T\sqcup\{k-1\}, T\sqcup\{k\}, S'\sqcup\{k\})$$
gives a homotopy between $e$ and a segment $(T\sqcup\{k-1\}, T\sqcup\{k\})$ in $Y$, followed by a segment in
$j_k(X(n-1,k-1))$.
Since $i_k(X(n,k-1))$, $j_k(X(n-1,k-1))$ and $Y$ are simply connected, this implies that $X(n,k)$ is also simply connected.
\end{proof}

\begin{cor} The algebra $\cA^{\circ\circ}$ has a structure of $\k[x_1,\ldots,x_k]$-algebra. In the case $n<k$, it has a structure
of an algebra over $R=\k[x_1,\ldots,x_k]/(x_1\ldots x_k)$.
\end{cor}

\begin{rmk}\label{OS-rem} Comparing our description of the algebra $\cA^{\circ\circ}$ in Theorem \ref{A-side-morphisms-general-thm}
with the definition of the Ozsv\'ath-Szab\'o's bordered algebras $\BB(m,k)$ in \cite[Sec.\ 3.2]{OS}, we see that there is an isomorphism
$$\cA^{\circ\circ}\simeq \BB(k,n),$$
which is analogous to the Auroux's identifications in \cite{aurouxggt} in the case of surfaces with one boundary component.
This raises a natural question of interpreting the bimodules defined in \cite{OS} in terms of Lagrangian correspondences.
\end{rmk}

\subsection{Localization by stops}

By the general stop removal theorem \cite[Theorem 1.16]{GPS2}, removing a stop corresponds to localizing by the subcategory of Lagrangians supported near the image of the stop. 
In our case the stops $\La_1$ and $\La_2$ are neighborhoods of the symplectic submanifolds $M_{n-1,k}$ and 
using the generators of the partially wrapped Fukaya category of $M_{n-1,k}$, we obtain that
these subcategories are given by
\[ \mathcal{D}_1 = \langle T_1 \times X : X = L_{i_1} \times \ldots \times L_{i_{n-1}} \rangle \]
\[ \mathcal{D}_2 = \langle T_2 \times X : X=L_{i_1} \times \ldots \times L_{i_{n-1}} \rangle   \]
where $\{i_1,\ldots, i_{n-1}\}$ runs through size $(n-1)$ subsets of $[0,k]$. 
Thus, we have the following result.

\begin{thm}\label{A-localization-thm} 
We have equivalences of pre-triangulated categories
\[ \mathcal{W}(M_{n,k}, \Lambda_{1}) \simeq \mathcal{W}(M_{n,k}, \Lambda) / \mathcal{D}_2 \]
\[  \mathcal{W}(M_{n,k}) \simeq \mathcal{W}(M_{n,k}, \Lambda) / \langle \mathcal{D}_1, \mathcal{D}_2 \rangle \]
\end{thm}

\begin{rmk}
If we only consider one stop, say $q_1$, then the partially wrapped Fukaya category $\WW(\Sigma,q_1)$ is generated
by the Lagrangians $L_1,\ldots, L_k$ (i.e., we do not need $L_0$), as in this case the complement of $L_1,\ldots,L_k$
is a disk which contains at most 1 stop in its boundary. 
Furthermore, by Auroux's theorem \cite[Theorem 1]{aurouxicm},
the category $\WW(M_{n-1,k},\Lambda_1)$ is generated by $L_{i_1}\times \ldots \times L_{i_{n-1}}$ with 
$\{i_1,\ldots, i_{n-1}\}\sub [1,k]$.
This allows to reduce the number of generators in the definition of $\DD_1$ and $\DD_2$ above.
For the same reason, it is enough to use only products of $L_1,\ldots,L_k$ when generating the fully 
wrapped Fukaya category of $M_{n,k}$. 
\end{rmk}

We next give an explicit description of the objects in $\mathcal{D}_1$ and $\mathcal{D}_2$ in terms of the generators of $\mathcal{W}(M_{n,k}, \Lambda_{Z})$. 

\begin{prop}\label{stop-resolutions-prop} 
Let $X = L_{i_1} \times \ldots \times L_{i_{n-1}}$ and let $j_1 < j_2 \ldots < j_{k-n+2}$ be the complement of $\{i_1, \ldots, i_{n-1}\}$ in $[0,k]$, then in $\mathcal{W}(M_{n,k},\Lambda_{Z})$ we have the following equivalences
\begin{align*}
T_1 \times X &\simeq \{ L_{j_1} \times X \rTo{u_{[j_1,j_2]}} L_{j_2} \times X \to \ldots 
\rTo{u_{[j_{k-n+1},j_{k-n+2}]}} L_{j_{k-n+2}} \times X \}, \\ 
T_2 \times X &\simeq \{ L_{j_{k-n+2}} \times X \rTo{v_{[j_{k-n+1},j_{k-n+2}]}} L_{j_{k-n+1}} \times X \to \ldots \rTo{v_{[j_1,j_2]}} L_{j_1} \times X \} 
\end{align*} 
The first equivalence also holds in $\mathcal{W}(M_{n,k}, \Lambda_{1})$. 
\end{prop} 

\begin{proof} This is obtained by repeatedly applying the exact triangle given by \eqref{exacttri}. More precisely, first as in Figure \ref{hamiso}, we slide $L_{j_{k+1}}$ along the components of $X$ until one of its legs gets next to $L_{j_k}$. This gives an 
isomorphism
$$L_{j_{k+1}} \times X \simeq \tilde{L}_{j_{k+1}} \times X.$$ 
Next, let us consider 
$$C: = \mathrm{cone}(\tilde{L}_{j_{k+1}} \times X \to L_{j_k} \times X).$$ 
Because of the above isomorphism
we also have an exact triangle
$$L_{j_{k+1}} \times X \to L_{j_k} \times X\to C\to\ldots$$
which is shown in Figure \ref{exacttrifig}.

Thus, applying this kind of triangle repeatedly, we can express $\tilde{T}_2 \times X$ drawn in Figure \ref{hamisotopy} as
\[ \tilde{T}_2 \times X \simeq \{ L_{j_{k-n+2}} \times X \to L_{j_{k-n+1}} \times X \to \ldots \to L_{j_1} \times X \}  \]
where $\tilde{T}_2$ has the property that all objects $L_0,\ldots, L_i$ to the left of the left leg of $\tilde{T}_2$ are in $X$ and all the objects $L_j, \ldots, L_k$ to the right of the right leg of $\tilde{T}_2$ are in $X$. Therefore, sliding $\tilde{T}_2$ over these Lagrangians, we exhibit a Hamiltonian isotopy between $\tilde{T}_2 \times X$and $T_2 \times X$, which completes the proof of the proposition for $T_2 \times X$. The case of $T_1 \times X$ is considered similarly.
\end{proof} 

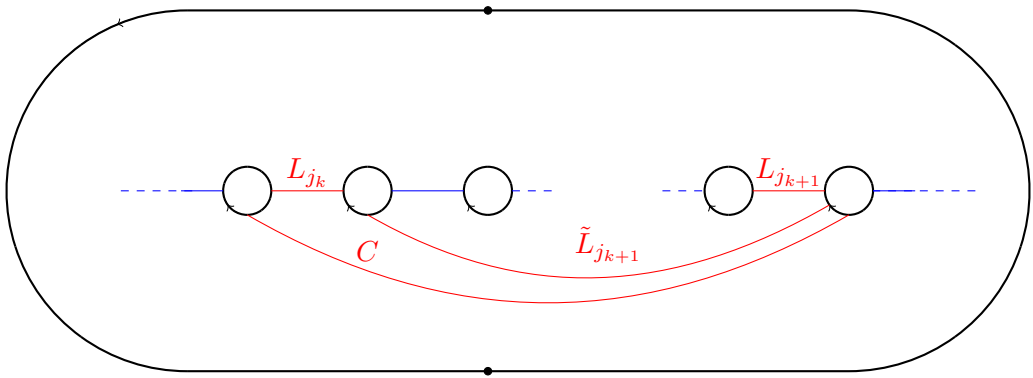
\begin{figure}[ht!]
\centering
\begin{tikzpicture}

\begin{scope}[scale=0.8]

\tikzset{
  with arrows/.style={
    decoration={ markings,
      mark=at position #1 with {\arrow{>}}
    }, postaction={decorate}
  }, with arrows/.default=2mm,
} 

\tikzset{vertex/.style = {style=circle,draw, fill,  minimum size = 2pt,inner        sep=1pt}}
\def \radius {1.5cm}

\foreach \s in {2.5,3.5,4.5,5.5,6.5} {
  
   \draw[thick] ([shift=({360/10*(\s)}:\radius+1.5cm)]0,0) arc ({360/10 *(\s)}:{360/10*(\s+1)}:\radius+1.5cm);
}

\foreach \s in {7.5,8.5,9.5,10.5,11.5} {
  
   \draw[thick] ([shift=({360/10*(\s)}:\radius+1.5cm)]11,0) arc ({360/10 *(\s)}:{360/10*(\s+1)}:\radius+1.5cm);
}

\draw[thick] (0,3) -- (11,3);

\draw[thick] (0,-3) -- (11,-3);

\foreach \s in {0.5,1.5,2.5,3.5,4.5,5.5,6.5,7.5,8.5,9.5} {
 
    \draw[thick] ([shift=({360/10*(\s)}:\radius-1.1cm)]1,0) arc ({360/10 *(\s)}:{360/10*(\s+1)}:\radius-1.1cm);

    \draw[thick] ([shift=({360/10*(\s)}:\radius-1.1cm)]3,0) arc ({360/10 *(\s)}:{360/10*(\s+1)}:\radius-1.1cm);
    \draw[thick] ([shift=({360/10*(\s)}:\radius-1.1cm)]5,0) arc ({360/10 *(\s)}:{360/10*(\s+1)}:\radius-1.1cm);
  \draw[thick] ([shift=({360/10*(\s)}:\radius-1.1cm)]9,0) arc ({360/10 *(\s)}:{360/10*(\s+1)}:\radius-1.1cm);
  \draw[thick] ([shift=({360/10*(\s)}:\radius-1.1cm)]11,0) arc ({360/10 *(\s)}:{360/10*(\s+1)}:\radius-1.1cm);

}

\draw[with arrows] ([shift=({360/10*(7)}:\radius-1.1cm)]1,0) arc ({360/10*(7)}:{360/10*(5)}:\radius-1.1cm);

\draw[with arrows] ([shift=({360/10*(7)}:\radius-1.1cm)]3,0) arc ({360/10*(7)}:{360/10*(5)}:\radius-1.1cm);

\draw[with arrows] ([shift=({360/10*(7)}:\radius-1.1cm)]5,0) arc ({360/10*(7)}:{360/10*(5)}:\radius-1.1cm);

\draw[with arrows] ([shift=({360/10*(7)}:\radius-1.1cm)]9,0) arc ({360/10*(7)}:{360/10*(5)}:\radius-1.1cm);
 
\draw[with arrows] ([shift=({360/10*(7)}:\radius-1.1cm)]11,0) arc ({360/10*(7)}:{360/10*(5)}:\radius-1.1cm);

\draw[with arrows]({360/10 * (3)}:\radius+1.5cm) arc ({360/10 *(3)}:{360/10*(4)}:\radius+1.5cm);

\draw [blue, dashed] ([shift=({360/10*(10)}:\radius-1.1cm)]-1.5,0) --  ([shift=({360/10*(5)}:\radius-1.1cm)]0.5,0); 

\draw [blue] ([shift=({360/10*(10)}:\radius-1.1cm)]-0.5,0) --  ([shift=({360/10*(5)}:\radius-1.1cm)]1,0); 

\draw [red] ([shift=({360/10*(10)}:\radius-1.1cm)]1,0) -- ([shift=({360/10*(10)}:\radius+1.1cm)]0,0);  

\draw [blue] ([shift=({360/10*(10)}:\radius-1.1cm)]3,0) -- ([shift=({360/10*(10)}:\radius+1.1cm)]2,0);  

\draw [blue, dashed] ([shift=({360/10*(10)}:\radius-1.1cm)]5,0) -- ([shift=({360/10*(10)}:\radius+1.1cm)]3.5,0);

\draw [blue, dashed] ([shift=({360/10*(10)}:\radius-1.1cm)]7.5,0) --  ([shift=({360/10*(5)}:\radius-1.1cm)]9,0);

\draw [red] ([shift=({360/10*(10)}:\radius-1.1cm)]9,0) -- ([shift=({360/10*(10)}:\radius+1.1cm)]8,0);  

\draw [blue] ([shift=({360/10*(10)}:\radius-1.1cm)]11,0) -- ([shift=({360/10*(10)}:\radius+1.1cm)]9.5,0);  

\draw [blue, dashed] ([shift=({360/10*(10)}:\radius+1.1cm)]10.5,0) -- ([shift=({360/10*(10)}:\radius-1.1cm)]11,0);  

\draw [red] ([shift=({360/10*(7.5)}:\radius-1.1cm)]3,0) to[in=210, out=330] ([shift=({360/10*(6)}:\radius-1.1cm)]11,0);

\draw [red] ([shift=({360/10*(7.5)}:\radius-1.1cm)]1,0) to[in=210, out=330] ([shift=({360/10*(7.5)}:\radius-1.1cm)]11,0);  

\node[red] at (3, -1) {\small $C$};
\node[red] at (2,0.3) {\small $L_{j_k}$};
\node[red] at (10,0.3) {\small $L_{j_{k+1}}$};
\node[red] at (7,-0.9) {\small $\tilde{L}_{j_{k+1}}$};

\node[vertex] at  ([shift=({360/10*(7.5)}:\radius+1.5cm)]5,0) {};
\node[vertex] at  ([shift=({360/10*(2.5)}:\radius+1.5cm)]5,0) {};

\end{scope}

\end{tikzpicture}
    \caption{Hamiltonian isotopy of $L_{j_{k+1}} \times X \to \tilde{L}_{j_{k+1}} \times X$}
    \label{hamiso}
\end{figure}

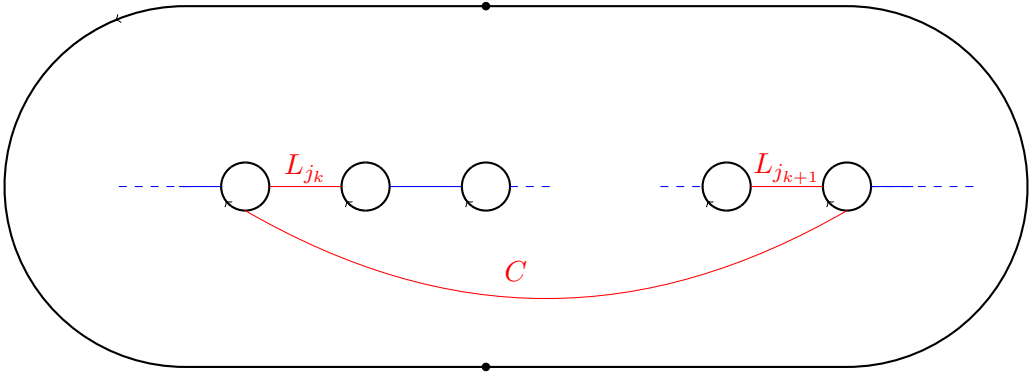
\begin{figure}[ht!]
\centering
\begin{tikzpicture}

\begin{scope}[scale=0.8]

\tikzset{
  with arrows/.style={
    decoration={ markings,
      mark=at position #1 with {\arrow{>}}
    }, postaction={decorate}
  }, with arrows/.default=2mm,
} 

\tikzset{vertex/.style = {style=circle,draw, fill,  minimum size = 2pt,inner        sep=1pt}}
\def \radius {1.5cm}

\foreach \s in {2.5,3.5,4.5,5.5,6.5} {
  
   \draw[thick] ([shift=({360/10*(\s)}:\radius+1.5cm)]0,0) arc ({360/10 *(\s)}:{360/10*(\s+1)}:\radius+1.5cm);
}

\foreach \s in {7.5,8.5,9.5,10.5,11.5} {
  
   \draw[thick] ([shift=({360/10*(\s)}:\radius+1.5cm)]11,0) arc ({360/10 *(\s)}:{360/10*(\s+1)}:\radius+1.5cm);
}

\draw[thick] (0,3) -- (11,3);

\draw[thick] (0,-3) -- (11,-3);

\foreach \s in {0.5,1.5,2.5,3.5,4.5,5.5,6.5,7.5,8.5,9.5} {
 
    \draw[thick] ([shift=({360/10*(\s)}:\radius-1.1cm)]1,0) arc ({360/10 *(\s)}:{360/10*(\s+1)}:\radius-1.1cm);

    \draw[thick] ([shift=({360/10*(\s)}:\radius-1.1cm)]3,0) arc ({360/10 *(\s)}:{360/10*(\s+1)}:\radius-1.1cm);
    \draw[thick] ([shift=({360/10*(\s)}:\radius-1.1cm)]5,0) arc ({360/10 *(\s)}:{360/10*(\s+1)}:\radius-1.1cm);
  \draw[thick] ([shift=({360/10*(\s)}:\radius-1.1cm)]9,0) arc ({360/10 *(\s)}:{360/10*(\s+1)}:\radius-1.1cm);
  \draw[thick] ([shift=({360/10*(\s)}:\radius-1.1cm)]11,0) arc ({360/10 *(\s)}:{360/10*(\s+1)}:\radius-1.1cm);

}

\draw[with arrows] ([shift=({360/10*(7)}:\radius-1.1cm)]1,0) arc ({360/10*(7)}:{360/10*(5)}:\radius-1.1cm);

\draw[with arrows] ([shift=({360/10*(7)}:\radius-1.1cm)]3,0) arc ({360/10*(7)}:{360/10*(5)}:\radius-1.1cm);

\draw[with arrows] ([shift=({360/10*(7)}:\radius-1.1cm)]5,0) arc ({360/10*(7)}:{360/10*(5)}:\radius-1.1cm);

\draw[with arrows] ([shift=({360/10*(7)}:\radius-1.1cm)]9,0) arc ({360/10*(7)}:{360/10*(5)}:\radius-1.1cm);
 
\draw[with arrows] ([shift=({360/10*(7)}:\radius-1.1cm)]11,0) arc ({360/10*(7)}:{360/10*(5)}:\radius-1.1cm);

\draw[with arrows]({360/10 * (3)}:\radius+1.5cm) arc ({360/10 *(3)}:{360/10*(4)}:\radius+1.5cm);

\draw [blue, dashed] ([shift=({360/10*(10)}:\radius-1.1cm)]-1.5,0) --  ([shift=({360/10*(5)}:\radius-1.1cm)]0.5,0); 

\draw [blue] ([shift=({360/10*(10)}:\radius-1.1cm)]-0.5,0) --  ([shift=({360/10*(5)}:\radius-1.1cm)]1,0); 

\draw [red] ([shift=({360/10*(10)}:\radius-1.1cm)]1,0) -- ([shift=({360/10*(10)}:\radius+1.1cm)]0,0);  

\draw [blue] ([shift=({360/10*(10)}:\radius-1.1cm)]3,0) -- ([shift=({360/10*(10)}:\radius+1.1cm)]2,0);  

\draw [blue, dashed] ([shift=({360/10*(10)}:\radius-1.1cm)]5,0) -- ([shift=({360/10*(10)}:\radius+1.1cm)]3.5,0);

\draw [blue, dashed] ([shift=({360/10*(10)}:\radius-1.1cm)]7.5,0) --  ([shift=({360/10*(5)}:\radius-1.1cm)]9,0);

\draw [red] ([shift=({360/10*(10)}:\radius-1.1cm)]9,0) -- ([shift=({360/10*(10)}:\radius+1.1cm)]8,0);  

\draw [blue] ([shift=({360/10*(10)}:\radius-1.1cm)]11,0) -- ([shift=({360/10*(10)}:\radius+1.1cm)]9.5,0);  

\draw [blue, dashed] ([shift=({360/10*(10)}:\radius+1.1cm)]10.5,0) -- ([shift=({360/10*(10)}:\radius-1.1cm)]11,0);

\draw [red] ([shift=({360/10*(7.5)}:\radius-1.1cm)]1,0) to[in=210, out=330] ([shift=({360/10*(7.5)}:\radius-1.1cm)]11,0);  
\


\node[red] at (2,0.3) {\small $L_{j_k}$};
\node[red] at (10,0.3) {\small $L_{j_{k+1}}$};




\node[vertex] at  ([shift=({360/10*(7.5)}:\radius+1.5cm)]5,0) {};
\node[vertex] at  ([shift=({360/10*(2.5)}:\radius+1.5cm)]5,0) {};

\node[red] at (5.5, -1.4) {\small $C$};



\end{scope}

\end{tikzpicture}
    \caption{Exact triangle}
    \label{exacttrifig}
\end{figure}

\begin{figure}[ht!]
\centering
\begin{tikzpicture}

\begin{scope}[scale=0.8]

\tikzset{
  with arrows/.style={
    decoration={ markings,
      mark=at position #1 with {\arrow{>}}
    }, postaction={decorate}
  }, with arrows/.default=2mm,
} 

\tikzset{vertex/.style = {style=circle,draw, fill,  minimum size = 2pt,inner        sep=1pt}}
\def \radius {1.5cm}

\foreach \s in {2.5,3.5,4.5,5.5,6.5} {
  
   \draw[thick] ([shift=({360/10*(\s)}:\radius+1.5cm)]0,0) arc ({360/10 *(\s)}:{360/10*(\s+1)}:\radius+1.5cm);
}

\foreach \s in {7.5,8.5,9.5,10.5,11.5} {
  
   \draw[thick] ([shift=({360/10*(\s)}:\radius+1.5cm)]11,0) arc ({360/10 *(\s)}:{360/10*(\s+1)}:\radius+1.5cm);
}

\draw[thick] (0,3) -- (11,3);

\draw[thick] (0,-3) -- (11,-3);

\foreach \s in {0.5,1.5,2.5,3.5,4.5,5.5,6.5,7.5,8.5,9.5} {
 
    \draw[thick] ([shift=({360/10*(\s)}:\radius-1.1cm)]1,0) arc ({360/10 *(\s)}:{360/10*(\s+1)}:\radius-1.1cm);

    \draw[thick] ([shift=({360/10*(\s)}:\radius-1.1cm)]-1,0) arc ({360/10 *(\s)}:{360/10*(\s+1)}:\radius-1.1cm);
  \draw[thick] ([shift=({360/10*(\s)}:\radius-1.1cm)]12,0) arc ({360/10 *(\s)}:{360/10*(\s+1)}:\radius-1.1cm);
  \draw[thick] ([shift=({360/10*(\s)}:\radius-1.1cm)]10,0) arc ({360/10 *(\s)}:{360/10*(\s+1)}:\radius-1.1cm);

}

\draw[with arrows] ([shift=({360/10*(7)}:\radius-1.1cm)]1,0) arc ({360/10*(7)}:{360/10*(5)}:\radius-1.1cm);

\draw[with arrows] ([shift=({360/10*(7)}:\radius-1.1cm)]-1,0) arc ({360/10*(7)}:{360/10*(5)}:\radius-1.1cm);

\draw[with arrows] ([shift=({360/10*(7)}:\radius-1.1cm)]10,0) arc ({360/10*(7)}:{360/10*(5)}:\radius-1.1cm);

\draw[with arrows] ([shift=({360/10*(7)}:\radius-1.1cm)]12,0) arc ({360/10*(7)}:{360/10*(5)}:\radius-1.1cm);
 
\draw[with arrows]({360/10 * (3)}:\radius+1.5cm) arc ({360/10 *(3)}:{360/10*(4)}:\radius+1.5cm);

\draw [blue, dashed] ([shift=({360/10*(10)}:\radius-1.1cm)]0.2,0) --  ([shift=({360/10*(5)}:\radius-1.1cm)]0.5,0); 

\draw [blue, dashed] ([shift=({360/10*(10)}:\radius-1.1cm)]-1,0) --  ([shift=({360/10*(5)}:\radius-1.1cm)]0.2,0); 

\draw [blue] ([shift=({360/10*(10)}:\radius-1.1cm)]-3.4,0) --  ([shift=({360/10*(5)}:\radius-1.1cm)]-1,0); 

\draw [blue, dashed] ([shift=({360/10*(10)}:\radius-1.1cm)]11.2,0) --  ([shift=({360/10*(5)}:\radius-1.1cm)]11.5,0); 

\draw [blue, dashed] ([shift=({360/10*(10)}:\radius-1.1cm)]10,0) --  ([shift=({360/10*(5)}:\radius-1.1cm)]11.2,0); 

\draw [black, dashed] ([shift=({360/10*(10)}:\radius-1.1cm)]6.2,0) --  ([shift=({360/10*(5)}:\radius-1.1cm)]3.7,0);

\draw [blue] ([shift=({360/10*(10)}:\radius-1.1cm)]12,0) --  ([shift=({360/10*(5)}:\radius-1.1cm)]14.4,0);

\draw [red] ([shift=({360/10*(7.5)}:\radius-1.1cm)]1,0) to[in=210, out=330] ([shift=({360/10*(7.5)}:\radius-1.1cm)]10,0);  

\node[blue] at (-2,0.3) {\small $L_{0}$};
\node[blue] at (13,0.3) {\small $L_{k}$};

\draw[purple] (4.8, -3) to[in=180,out=90] (5,-2.7);
\draw[purple] (5, -2.7) to[in=90,out=0] (5.2,-3);

\node[purple] at ([shift=({360/10*(7)}:\radius+1.5cm)]5.5,0.4) {\small $T_2$}; 

\node[red] at (7,-0.9) {\small $\tilde{T}_2$};

\node[vertex] at  ([shift=({360/10*(7.5)}:\radius+1.5cm)]5,0) {};
\node[vertex] at  ([shift=({360/10*(2.5)}:\radius+1.5cm)]5,0) {};



\end{scope}

\end{tikzpicture}
    \caption{Hamiltonian isotopy $T_2 \times X \to \tilde{T}_2 \times X$}
    \label{hamisotopy}
\end{figure}
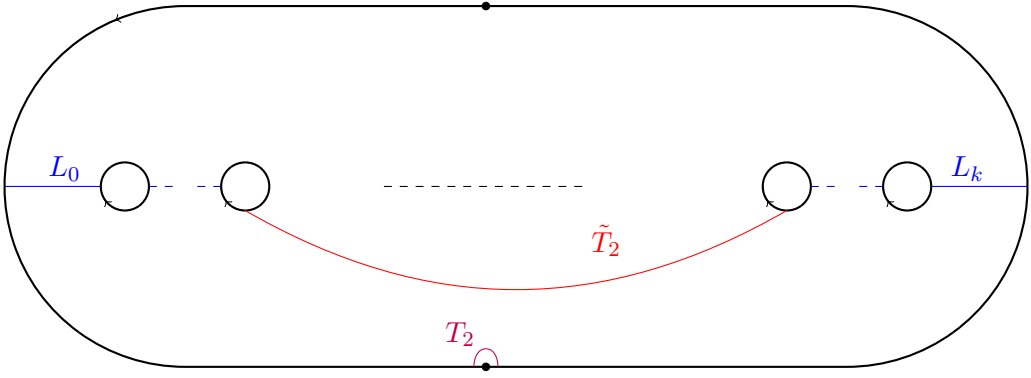

\section{B-side}\label{B-sec}
\subsection{Two categorical resolutions on the B-side}

Recall that we work over a base commutative ring $\mathbf{k}$.
Let $R=R_{[1,k]}$ denote the ring $\mathbf{k}[x_1,\ldots,x_k]/(x_1\ldots x_k)$.
Here we define two categorical resolutions of the category $\Perf(R)$.
Both are given using modules over certain finite $R$-algebras.

For every subset $I\sub [1,k]$, let us set $x_I:=\prod_{i\in I}x_i$.
We consider two $R$-algebras, 
$$\BB^{\circ}=\BB^{\circ}_{[1,k]}:=\End_R(R/(x_1)\oplus R/(x_{[1,2]})\oplus \ldots \oplus R/(x_{[1,k-1]})\oplus R),$$
$$\BB^{\circ\circ}=\BB^{\circ\circ}_{[1,k]}:=\End_R(\bigoplus_{I\sub [1,k],I\neq\emptyset}R/(x_I)),$$
where the summation is over all nonempty subintervals of $[1,k]$. 
Note that both $\BB^\circ$ and $\BB^{\circ\circ}$ are finitely generated as $R$-modules
(in particular, they are Noetherian when $\k$ is Noetherian).
  
For each nonempty subinterval $I\sub [1,k]$ (resp., for $I=[1,m]$) we denote by $P_I$ the natural projective
${\BB^{\circ\circ}}$-module (resp., $\BB^{\circ}$-module) corresponding to the summand $R/(x_I)$. 
When we need to distinguish the $\BB^{\circ}$-module
$P_{[1,m]}$ from the ${\BB^{\circ\circ}}$-module with the same name, we will write $P_{[1,m]}^{\circ}$ 
(resp., $P_{[1,m]}^{\circ\circ}$) to denote the
${\BB^{\circ}}$-module (resp., ${\BB^{\circ\circ}}$-module).

Note that every ${\BB^{\circ\circ}}$-module $M$, viewed as an $R$-module, has a decomposition
$$M=\bigoplus_{I\sub [1,k],I\neq\emptyset} M_I,$$
with $M_I=\Hom_{\BB^{\circ\circ}}(P_I,M)$,
corresponding to the natural idempotents in ${\BB^{\circ\circ}}$.  

It is easy to see that we have natural identifications
$$\Hom(P_I,P_J)\simeq\Hom_R(R/(x_J),R/(x_I)),$$
compatible with composition.
In particular, $\End(P_I)=R/(x_I)$, and
$$({\BB^{\circ}})^{op}=\End_{\BB^{\circ\circ}}(\bigoplus_{m=1}^k P_{[1,m]}).$$

Thus, if we set $P=P_{[1,k]}$, then we get natural faithful functors 
$$i_R^{\BB^{\circ}}:\Perf(R)\to \Perf({\BB^{\circ}}): M\mapsto P\ot_R M,$$
$$i_{\BB^{\circ}}^{\BB^{\circ\circ}}:\Perf({\BB^{\circ}})\to \Perf({\BB^{\circ\circ}}):M\mapsto 
(\bigoplus_{m=1}^k P_{[1,m]})\ot_{\BB^{\circ}} M,$$
which are left adjoint to the restriction functors
$$r^{\BB^{\circ}}_R:D({\BB^{\circ}})\to D(R): M\mapsto\Hom_{\BB^{\circ}}(P,M),$$
\begin{equation}\label{resfun1-eq}
r^{\BB^{\circ\circ}}_{\BB^{\circ}}:D({\BB^{\circ\circ}})\to D({\BB^{\circ}}): M\mapsto\Hom_{\BB^{\circ\circ}}(\bigoplus_{m=1}^k P_{[1,m]},M).
\end{equation}
Here we denote by $D(A)$ the bounded derived category of $A$-modules,
while $\Perf(A)$ is the full subcategory of bounded complexes of finitely generated projective modules.
We also consider composed functors $i_R^{\BB^{\circ\circ}}=i_{\BB^{\circ}}^{\BB^{\circ\circ}}\circ i_R^{\BB^{\circ}}$ and 
\begin{equation}\label{resfun2-eq}
r^{\BB^{\circ\circ}}_R=r^{\BB^{\circ}}_R\circ r^{\BB^{\circ\circ}}_{\BB^{\circ}}.
\end{equation}

We are going to show that the algebras ${\BB^{\circ\circ}}$ and ${\BB^{\circ}}$ are homologically smooth over $\k$ and that
the restriction functors $r^{\BB^{\circ\circ}}_{\BB^{\circ}}$ and $r^{\BB^{\circ\circ}}_R$ are localization functors with respect to explicit subcategories
of $D^b({\BB^{\circ\circ}})$ (provided $\k$ is regular), 
whereas $i_R^{\BB^{\circ}}$ and $i_R^{\BB^{\circ\circ}}$ are fully faithful, so that ${\BB^{\circ\circ}}$ and ${\BB^{\circ}}$ provide categorical resolutions
of $R$.

For a subinterval $I\sub [1,k]$,
let us denote by $R_I$, $\BB^{\circ\circ}_I$ and $\BB^{\circ}_I$ the algebras defined in the same way
as $R_{[1,k]}$, $\BB^{\circ\circ}_{[1,k]}$ and $\BB^{\circ}_{[1,k]}$, but with the variables $x_1,\ldots,x_k$ replaced by $(x_i)_{i\in I}$.

We will need to use some other induction functors, in addition to $i^{\BB^{\circ\circ}}_{\BB^{\circ}}$ and $i^{\BB^{\circ}}_R$.
Let us consider the natural (exact) restriction functors
$$r^{\BB^{\circ\circ}}_{\BB^{\circ\circ}_{[1,k-1]}[x_k]}:{\BB^{\circ\circ}}-\mod\to \BB^{\circ\circ}_{[1,k-1]}[x_k]-\mod: M\mapsto \Hom_{\BB^{\circ\circ}}(\bigoplus_{J\sub [1,k-1],J\neq\emptyset}P_J,M),$$
$$r^{\BB^{\circ\circ}}_{\BB^{\circ\circ}_{[2,k]}[x_1]}:{\BB^{\circ\circ}}-\mod\to \BB^{\circ\circ}_{[2,k]}[x_1]-\mod: M\mapsto \Hom_{\BB^{\circ\circ}}(\bigoplus_{J\sub [2,k],J\neq\emptyset}P_J,M),$$
where we use the identifications
$$R_{[1,k]}/x_{[1,k-1]}=R_{[1,k-1]}[x_k], \ \ R_{[1,k]}/x_{[2,k]}=R_{[2,k]}[x_1].$$
It is easy to see that in both cases the restricted module has the decomposition
into components (which are $R$-modules) of the form
$$r^{\BB^{\circ\circ}}_{\BB^{\circ\circ}_I}(M)=\bigoplus_{J\sub I,J\neq\emptyset} M_J.$$

\begin{lem}\label{res-ind-lem} 
(i) One has natural isomorphisms
$$i^{\BB^{\circ\circ}}_{\BB^{\circ}}(P^\circ_{[1,i]})\simeq P^{\circ\circ}_{[1,i]}, \ \ i^{\BB^{\circ}}_R(R)\simeq P=P_{[1,k]},$$
where $1\le i\le k$.
The canonical adjunction maps
$$\Id\to r^{\BB^{\circ\circ}}_{\BB^{\circ}}\circ i^{\BB^{\circ\circ}}_{\BB^{\circ}}, \ \ \Id\to r^{\BB^{\circ}}_R\circ i^{\BB^{\circ}}_R$$
are isomorphisms on perfect derived categories and on abelian categories.

\noindent
(ii) The restriction functor $r^{\BB^{\circ\circ}}_{\BB^{\circ\circ}_{[2,k]}[x_1]}$ has a left adjoint functor
$$i^{\BB^{\circ\circ}}_{\BB^{\circ\circ}_{[2,k]}[x_1]}:\BB^{\circ\circ}_{[2,k]}-\mod\to {\BB^{\circ\circ}}-\mod$$
such that $r^{\BB^{\circ\circ}}_{\BB^{\circ\circ}_{[2,k]}[x_1]}\circ i^{\BB^{\circ\circ}}_{\BB^{\circ\circ}_{[2,k]}[x_1]}\simeq \Id$ and
$$i^{\BB^{\circ\circ}}_{\BB^{\circ\circ}_{[2,k]}[x_1]}(M)_I=\begin{cases} M_{I\cap [2,k]} & I\neq [1], \\ 0 & I=[1].\end{cases}$$
In particular,  the functor $i^{\BB^{\circ\circ}}_{\BB^{\circ\circ}_{[2,k]}[x_1]}$ is exact and 
$r^{\BB^{\circ\circ}}_{\BB^{\circ\circ}_{[2,k]}[x_1]}\circ i^{\BB^{\circ\circ}}_{\BB^{\circ\circ}_{[2,k]}[x_1]}=\Id$.
For $I\sub [2,k]$, one has
$$i^{\BB^{\circ\circ}}_{\BB^{\circ\circ}_{[2,k]}[x_1]}(P_{[2,k]}[x_1])\simeq P_{[2,k]}.$$

\noindent
(iii) The restriction functor $r^{\BB^{\circ}}_{\BB^{\circ}_{[1,k-1]}[x_k]}:{\BB^{\circ}}-\mod\to \BB^{\circ}_{[1,k-1]}[x_k]-\mod$ has a left
adjoint functor 
$$i^{\BB^{\circ}}_{\BB^{\circ}_{[1,k-1]}[x_k]}:\BB^{\circ}_{[1,k-1]}[x_k]-\mod\to {\BB^{\circ}}-\mod,$$
which is exact and satisfies $r^{\BB^{\circ}}_{\BB^{\circ}_{[1,k-1]}[x_k]}\circ i^{\BB^{\circ}}_{\BB^{\circ}_{[1,k-1]}[x_k]}=\Id$.
Furthermore, for $1\le i\le k-1$, we have
$$i^{\BB^{\circ}}_{\BB^{\circ}_{[1,k-1]}[x_k]}(P_{[1,i]}[x_k])\simeq P_{[1,i]}.$$
\end{lem}

\Pf . (i) For any ${\BB^{\circ\circ}}$-module $M$, we have
$$\Hom_{B^{\circ\circ}}(P^{\circ\circ}_{[1,i]},M)\simeq M_{[1,i]}=r^{\BB^{\circ\circ}}_{\BB^{\circ}}(M)_{[1,i]}\simeq
\Hom_{\BB^{\circ}}(P^{\circ}_{[1,i]},r^{\BB^{\circ\circ}}_{\BB^{\circ}}(M)),$$
which gives an isomorphism $i^{\BB^{\circ\circ}}_{\BB^{\circ}}(P^{\circ}_{[1,i]})\simeq P^{\circ\circ}_{[1,i]}$.
The proof of the isomorphism $i^{\BB^{\circ}}_R(R)\simeq P$ is similar.

This implies that our two adjunction maps are isomorphisms on all projective modules, hence on the perfect derived subcategories. It is also easy to see that these maps are isomorphisms on the abelian categories (since the induction
functors on the abelian categories can be explicitly described).

\noindent
(ii),(iii). The proofs are straightforward and similar to (i).
\ed

For a pair of disjoint subintervals $I,J\sub [1,k]$, such that $I\sqcup J$ is again an interval, 
let us denote by 
$$\a\{I,J\}\in\Hom_{\BB^{\circ\circ}}(P_I,P_{I\sqcup J})=\Hom_R(R/x_{I\sqcup J},R/x_{I})$$
the generator corresponding to the natural projection $R/x_{I\sqcup J}\to R/x_{I}$,
and let
$$\b\{I,J\}\in\Hom_{\BB^{\circ\circ}}(P_{I\sqcup J},P_J)=\Hom_R(R/x_J,R/x_{I\sqcup J})$$
be the generator corresponding to the map $R/x_J\rTo{x_{I}} R/x_{I\sqcup J}$.
In the case when this makes sense we use the same notation for the similar morphisms in the category of ${\BB^{\circ}}$-modules.

\begin{prop}\label{B-semiorth-dec-prop}
For every $i=2,\ldots,k$ let us define a ${\BB^{\circ}}$-module $\ov{P}_{i}$ from the exact sequence
$$0\to P_{[1,i-1]}\rTo{\a\{[1,i-1],[i]\}} P_{[1,i]}\to \ov{P}_{i}\to 0,$$
and let $\ov{P}_{1}=P_{[1]}$.
Then
\begin{equation}\label{B-mod-Ext-eq}
\Ext^m_{\BB^{\circ}}(\ov{P}_{i},\ov{P}_{i})\simeq \begin{cases} R/(x_i) & m=0\\ 0 & m\neq 0,\end{cases}
\end{equation}
and one has a semiorthogonal decomposition,
$$\Perf({\BB^{\circ}})=\lan \lan \ov{P}_{k}\ran,\ldots,\lan \ov{P}_{2}\ran, \lan \ov{P}_1\ran\ran,$$
where $\lan \ov{P}_i\ran\simeq \Perf(R/(x_i))$.
\end{prop}

\Pf . The surjectivity of the map $R/(x_{[1,i]})\to R/(x_{[1,i-1]})$
implies that the map $\a\{[1,i-1],[i]\}$ is an embedding.
Furthermore, it is easy to see that
$$(\ov{P}_{i})_{[1,j]}=\begin{cases} R/(x_i), & j\ge i, \\ 0, & j<i. \end{cases}$$
In particular, $\Ext^*(P_{[1,j]},\ov{P}_{i})=0$ for $j<i$ and $\Ext^*(P_{i},\ov{P}_{i})=R/(x_i)$.
This immediately implies the required semiorthogonalities and the equality \eqref{B-mod-Ext-eq}.

Now let us consider the restriction and induction functors 
$$r^{\BB^{\circ}}_{\BB^{\circ}_{[1,k-1]}[x_k]}:\Perf({\BB^{\circ}})\to D^b(\BB^{\circ}_{[1,k-1]}[x_k]), \ \ i^{\BB^{\circ}}_{\BB^{\circ}_{[1,k-1]}[x_k]}:D^b(\BB^{\circ}_{[1,k-1]}[x_k])\to D^b({\BB^{\circ}}).$$
We have $r^{\BB^{\circ}}_{\BB^{\circ}_{[1,k-1]}[x_k]}\circ  i^{\BB^{\circ}}_{\BB^{\circ}_{[1,k-1]}[x_k]}=\Id$, so the functor
$i^{\BB^{\circ}}_{\BB^{\circ}_{[1,k-1]}[x_k]}$ is fully faithful. Since for $i\le k-1$ one has
\begin{equation}\label{ind-P-1i-eq}
i^{\BB^{\circ}}_{\BB^{\circ}_{[1,k-1]}[x_k]}(P_{[1,i]}[x_k])=P_{[1,i]},
\end{equation}
we deduce that this induction functor sends perfect complexes to perfect complexes and 
$$i^{\BB^{\circ}}_{\BB^{\circ}_{[1,k-1]}[x_k]}(\ov{P}_{i}[x_k])\simeq\ov{P}_{i},$$
for $i\le k-1$. Using the induction on $k$, we see that it remains to prove that there is a semiorthogonal decomposition
\begin{equation}\label{Bcirc-semiorth-dec-eq}
\Perf({\BB^{\circ}})=\lan \lan\ov{P}_{k}\ran, i^{\BB^{\circ}}_{\BB^{\circ}_{[1,k-1]}[x_k]}\Perf(\BB^{\circ}_{[1,k-1]}[x_k])\ran.
\end{equation}
Since we already proved the semiorthogonality, by \cite[Lem.\ 1.20]{BLL}, it is enough to check that the two subcategories on the right generate
$\Perf({\BB^{\circ}})$. In view of \eqref{ind-P-1i-eq}, it remains to see that they generate the module $P_{[1,k]}$. But this follows from the exact
sequence defining $\ov{P}_k$.
\ed

\begin{rmk}\label{ov-P-res-rem} It is easy to see that for every $i$ one has an isomorphism of $B^\circ$-modules
$$\ov{P}_i\simeq r^{B^{\circ\circ}}_{B^\circ}(P^{\circ\circ}_{[i]}).$$
\end{rmk}
                                                                                                                             
Now we introduce a certain family of ${\BB^{\circ\circ}}$-modules, which will play an important role for us.

\begin{lem}\label{M-mod-lem} For a pair of nonempty disjoint subintervals $I, J\sub [1,k]$, such that $I\sqcup J$
is still an interval, the sequence
$$0\to P_I\rTo{\a\{I,J\}} P_{I\sqcup J}\rTo{\b\{I,J\}} P_J$$ 
is exact. We define the module $M\{I,J\}$ as the cokernel of the last arrow, so that we have an exact sequence
$$0\to P_I\to P_{I\sqcup J}\to P_J\to M\{I,J\}\to 0.$$ 
\end{lem}

\Pf . We just need to show that the kernel of the morphism
$P_{I\sqcup J}\to P_{J}$ is isomorphic to $P_{I}$. But this immediately follows from
the exact sequence
$$0\to R/(x_{J})\rTo{x_{I}} R/(x_{I\sqcup J})\to R/(x_{I})\to 0.$$
\ed

For every $j\in [2,k]$, we have a natural exact functor of ``extending by zero",
$$i_!\{j\}:\BB^{\circ\circ}_{[j,k]}[x_1,\ldots,x_{j-2}]-\mod\to {\BB^{\circ\circ}}-\mod.$$
Namely, it is given by
$$(i_!\{j\}M)_I=\begin{cases} M_I & I\sub [j,k], \\ 0 & \text{otherwise}.\end{cases}$$
It is easy to see that there is a natural structure of a ${\BB^{\circ\circ}}$-module on $i_!\{j\}M$.

We will also use the notation $i_!\{1\}=\Id:{\BB^{\circ\circ}}-\mod\to{\BB^{\circ\circ}}-\mod$.

\begin{lem}\label{i!-res-lem} 
For every $m\ge j\ge 2$ one has an isomorphism
of ${\BB^{\circ\circ}}$-modules,
$$M\{[j-1],[j,m]\}\simeq i_!\{j\}(P_{[j,m]}[x_1,\ldots,x_{j-2}]).$$
\end{lem}

\Pf . We need to compute the cokernel $C$ of the map $\b\{[j-1],[j,m]\}:P_{[j-1,m]}\to P_{[j,m]}$.
We have
$$(P_{[j,m]})_I=\Hom(R/(x_{[j,m]}),R/(x_I))=(x_{I\setminus [j,m]})/(x_I), \ \ (P_{[j-1,m]})_I=(x_{I\setminus [j-1,m]})/(x_I),$$
and our map is given by the multiplication by $x_{j-1}$. 

Assume first that $j-1\not\in I$. Then $I\setminus [j,m]=I\setminus [j-1,m]$, so we get
$$C_I=(P_{[j,m]})_I\ot R/(x_i).$$
Note that in this case we have either $I\sub [1,j-1]$ or $I\sub [j,k]$, and in the former case, $(P_{[j,m]})_I=0$, so $C_I=0$.

On the other hand, if $j-1\in I$ then setting $I'=I\setminus [j-1,m]$ we have
$$(P_{[j,m]})_I=(x_{j-1}x_{I'})/(x_I), \ \ (P_{[j-1,m]})_I=x_{I'}/(x_I),$$
so in this case the map $(P_{[j-1,m]})_I\rTo{x_{j-1}} (P_{[j,m]})_I$ is surjective, so $C_I=0$.

This easily leads to the identification $C\simeq i_!\{j\}(P_{[j,m]}[x_1,\ldots,x_{j-2}])$.
\ed

\begin{thm}\label{A-semiorth-dec-thm}
For every $j=1,\ldots,k$, the composed functor
\begin{equation}\label{B-j-A-emb}
\Phi_j=i_!\{j\}i_{\BB^{\circ}_{[j,k]}}^{\BB^{\circ\circ}_{[j,k]}}:\Perf(\BB^{\circ}_{[j,k]}[x_1,\ldots,x_{j-2}]))\to D^b({\BB^{\circ\circ}})
\end{equation}
is fully faithful, takes values in $\Perf({\BB^{\circ\circ}})$, and we have a semiorthogonal decomposition
$$\Perf({\BB^{\circ\circ}})=
\lan \Phi_k \Perf(\BB^{\circ}_{[k]}[x_1,\ldots,x_{k-2}]), \ldots, \Phi_3 \Perf(\BB^{\circ}_{[3,k]}[x_1]), \Phi_2 \Perf(\BB^{\circ}_{[2,k]}), \Phi_1 \Perf(\BB^{\circ}_{[1,k]})\ran.$$
\end{thm}

\Pf . Recall that by Lemma \ref{res-ind-lem}(i),
the functor $i^{\BB^{\circ\circ}}_{\BB^{\circ}}:\Perf({\BB^{\circ}})\to \Perf({\BB^{\circ\circ}})$ is fully faithful. Thus, the functor $\Phi_1$ is fully faithful.

To check that the functor $\Phi_j$, for $j\ge 2$, is fully faithful and takes values in $\Perf({\BB^{\circ\circ}})$, we observe that by Lemma \ref{i!-res-lem}
together with Lemma \ref{res-ind-lem}(i), we have
\begin{equation}\label{Phi-j-main-formula}
\Phi_j(P^{\circ}_{[j,m]}[x_1,\ldots,x_{j-2}])=M\{[j-1],[j,m]\}.
\end{equation}
By definition, the latter module is perfect.
Also, for any $\BB^{\circ\circ}_{[j,k]}[x_1,\ldots,x_{j-2}]$-module $M$ one has
$\Hom(P_{[j-1,m]},i_!\{j\}M)=0$ for any $m\ge j-1$. Hence, using \eqref{Phi-j-main-formula} and the defining exact sequence
for $M\{[j-1],[j,m]\}$, we get an
identification
\begin{align*}
&\Hom(\Phi_j(P^{\circ}_{[j,m]}[x_1,\ldots,x_{j-2}]),i_!\{j\}M)\simeq \Hom(P_{[j,m]},i_!\{j\}M)\simeq
(i_!\{j\}M)_{[j,m]} \\
&=M_{[j,m]}\simeq\Hom(P_{[j,m]}[x_1,\ldots,x_{j-2}],M).
\end{align*}
Applying this to $M=i_{\BB^{\circ}_{[j,k]}}^{\BB^{\circ\circ}_{[j,k]}}(P^{\circ}_{[j,m']}[x_1,\ldots,x_{j-2}])$,
we derive that $\Phi_j$ is fully faithful.

For $j\le k$, let us denote by $\PP_j\sub \Perf({\BB^{\circ\circ}})$ the subcategory generated by 
the projective modules $P_{[j',m]}$ with $j'\le j$, so that $\PP_1=\Phi_1 \Perf(\BB^{\circ}_{[1,k]})$.
Let us prove by induction on $j$ that we have a semiorthogonal decomposition
\begin{equation}\label{Pj-semiorth-dec-eq}
\PP_j=\lan \im\Phi_j, \PP_{j-1}\ran
\end{equation}
(for $j=k$ this will give the theorem).
Note that \eqref{Phi-j-main-formula} together with the defining exact sequence
for $M\{[j-1],[j,m]\}$ imply that $\im\Phi_j\sub \PP_j$.
The semiorthogonality immediately follows from the fact
that the image of $i_!\{j\}$ (and hence of $\Phi_j$) is right orthogonal to $\PP_{j-1}$.
Finally, the defining exact sequence for $M\{[j-1],[j,m]\}$ shows that $P_{[j,m]}$ belongs to the subcategory generated by
$\im\Phi_j$ and $\PP_{j-1}$. Hence, our assertion follows from \cite[Lem.\ 1.20]{BLL}.
\ed


\begin{rmk} Noncommutative resolutions similar to $\BB^\circ$ and $\BB^{\circ\circ}$ were considered in
\cite{DFI} and in \cite{Kuz-Lunts}.
\end{rmk}

\subsection{Homological smoothness}

Recall that a $\k$-algebra $B$, flat over $\k$, is called {\it homologically smooth} if the diagonal bimodule $B$ is perfect as an object of $D(B\ot_{\k} B^{op})$.
In the case when $\k$ is a field, this implies that $B$ has finite homological dimension (\cite[Lem.\ 3.6(b)]{Lunts}).
A similar argument proves the following result.

\begin{lem}\label{fin-proj-dim-perfect-lem}
Let $B$ be a homologically smooth $\k$-algebra, flat over $\k$. Assume also that $B$ is left Noetherian. Then every finitely generated $B$-module $M$,
that has finite projective dimension as a $\k$-module, is perfect as a $B$-module.
\end{lem}

\Pf . We use the fact that the diagonal bimodule is perfect, so it is a homotopy direct summand of a bounded complex of free $B\ot_{\k} B^{op}$-modules of finite rank.
Applying the corresponding functors $D(B)\to D(B)$ to $M$, we get that $M$ is a homotopy direct summand in a bounded complex of $B$-modules
whose terms are finite direct sums of
$B\ot_{\k} M$. By assumption, these terms have finite projective dimension over $B$, so $M$ also has finite projective dimension over $B$. Since $M$ is finitely generated,
this implies that it is perfect.
\ed

\begin{prop}\label{hom-smooth-prop} Assume that $\k$ is Noetherian.
Then the algebras ${\BB^\circ}$ and ${\BB^{\circ\circ}}$ are homologically smooth over $\k$.
\end{prop}

Note that wrapped Fukaya categories are always homologically smooth. This can be seen by combining \cite[Thm 1.2]{GPS1} and \cite[Thm 1.2]{ganatra}. 
Thus, Proposition \ref{hom-smooth-prop} will follow from Theorem \ref{main-thm} below. However, it is also not hard to give a direct proof.
The idea of proof is to use semiorthogonal decompositions \eqref{Bcirc-semiorth-dec-eq} and \eqref{Pj-semiorth-dec-eq}, together with the following fact.
Here we use the notion of gluing of dg-categories along a bimodule, which reflects the situation of a semiorthogonal decomposition on the dg-level
(see \cite[Sec. 4]{Kuz-Lunts}).
  
\begin{lem}\label{Kuz-Lunts-lem} (\cite[Prop.\ 4.9]{Kuz-Lunts}) 
Let $\DD_1$ and $\DD_2$ be smooth dg-categories, flat over $\k$, and let $\DD=\DD_1\times_\varphi \DD_2$ be a dg-category
obtained by gluing $\DD_1$ and $\DD_2$ along a bimodule $\varphi$ in $\Perf(\DD_2^{op}\ot_{\k} \DD_1)$. Then $\DD$ is smooth over $\k$.
\end{lem}

The more concrete situation in which we will apply this is described in the following

\begin{lem}\label{smooth-semiorth-lem} 
Let $\DD$ be a dg-category over $\k$ equipped with a semiorthogonal decomposition
$$\Perf(\DD)=\lan \lan M_1\ran,\lan M_2\ran \ran,$$
where for $i=1,2$, $\Ext^{\neq 0}(M_i,M_i)=0$ while $B_i:=\End(M_i)$ is a homologically smooth $\k$-algebra $B_i$, flat over $\k$.
Assume in addition that $B_1$ is Noetherian and the complex $\hom(M_1,M_2)$ has bounded cohomology, finitely generated as $B_1$-modules
and of finite projective dimension as $\k$-modules.
Then $\DD$ is homologically smooth.
\end{lem}

\Pf . To apply Lemma \ref{Kuz-Lunts-lem} we need to check that the $B_1-B_2$-bimodule $\hom(M_1,M_2)$ is perfect. Since $B_2$ is smooth,
it is enough to check that it is perfect as a $B_1$-module (see \cite[Lem.\ 2.8.2]{TV}). By Lemma \ref{fin-proj-dim-perfect-lem}, this follows from our assumptions.
\ed

\noindent
{\it Proof of Proposition \ref{hom-smooth-prop}.}
To prove that ${\BB^\circ}$ is smooth we apply Lemma \ref{smooth-semiorth-lem} iteratively to semiorthogonal decompositions \eqref{Bcirc-semiorth-dec-eq}.
Since the algebras $R/(x_i)$ are homologically smooth, it suffices to check that  
the $R$-modules $\Ext^*(\ov{P}_k,P_{[1,i]})$, for $i<k$, are finitely generated and are free as $\k$-modules.
Computing these using the projective resolution of $\ov{P}_k$, we realize them as cohomology of the complex (in degrees $[0,1]$)
$$R/(x_{[1,i]})\rTo{x_{[i+1,k]}} R/x_{[1,i]}.$$
Thus, we only have nontrivial $\Ext^1$ isomorphic to $R/(x_{[1,i]},x_{[i+1,k]})$, and the assertion follows.

Similarly, to prove that ${\BB^{\circ\circ}}$ is smooth we apply Lemma \ref{smooth-semiorth-lem} iteratively to semiorthogonal decompositions \eqref{Pj-semiorth-dec-eq}.
It suffices to prove that the $R$-modules $\Ext^*(M\{[j-1],[j,m]\},P_I)$, where $I=[j',m']$ with $j'<j$, are finitely generated and are free as $\k$-modules.
Using the defining projective resolution of $M\{[j-1],[j,m]\}$, we get the complex (in degrees $[0,2]$) that computes this $\Ext^*$.
If $m'<j-1$ then this complex is zero. For $m'\ge j-1$ it has the form
$$R/(x_{[j,m]\cap I})\rTo{x_{j-1}}R/(x_{[j-1,m]\cap I})\rTo{x_{[j,m]\setminus I}}R/(x_{j-1}),$$
Thus, only $\Ext^2$ is nontrivial and it is isomorphic to $R/(x_{j-1},x_{[j,m]\setminus I})$, which is free over $\k$.
\ed

\subsection{Localization on the B-side}

We say that a collection of objects $(X_i)_{i\in I}$ in an abelian category $\CC$ {\it generates} $\CC$ if the minimal abelian
subcategory of $\CC$ closed under extensions and containing all $X_i$ is the entire $\CC$.

We want to describe the restriction functors $r^{\BB^{\circ\circ}}_{\BB^{\circ}}$ and
$r^{\BB^{\circ\circ}}_R$ (see \eqref{resfun1-eq} and \eqref{resfun2-eq}) as localizations with respect to some subcategories.

\begin{thm}\label{B-loc-thm} 
(i) The functor $r^{\BB^{\circ\circ}}_{\BB^{\circ}}$ induces an
equivalence of enhanced triangulated categories
$$\Perf(\BB^{\circ\circ})/\ker(r^{\BB^{\circ\circ}}_{\BB^{\circ}})\simeq \Perf(\BB^{\circ}),$$                        

For every $1\le i\le j<m\le k$, the module
$M\{[i,j],[j+1,m]\}$ belongs to the subcategory $\ker(r^{\BB^{\circ\circ}}_{\BB^{\circ}})\sub\Perf(\BB^{\circ\circ})$.
Furthermore, $\ker(r^{\BB^{\circ\circ}}_{\BB^{\circ}})$
is generated by the modules $(M\{[i],[i+1,j]\})_{i<j}$.


\noindent
(ii) Assume that the ${\mathbf k}$ is regular.
Then the functor $r^{\BB^{\circ\circ}}_R$ induces an 
equivalence of enhanced triangulated categories
$$D^b(\BB^{\circ\circ})/\ker(r^{\BB^{\circ\circ}}_R)\simeq D^b(R),$$
where $D^b(?)$ denotes the bounded derived category of finitely generated modules.  

For every pair of disjoint intervals $I,J\sub [1,k]$, such that $I\sqcup J$ is an interval,
the modules $M\{I,J\}$ belong to the subcategory $\ker(r^{\BB^{\circ\circ}}_R)$.
Furthermore, $\ker(r^{\BB^{\circ\circ}}_R)$
is generated by the modules $(M\{[i],[i+1,j]\},  M\{[j],[i,j-1]\})_{i<j}$ as a triangulated category.
\end{thm}

\Pf . (i) The first assertion follows immediately from the
semiorthogonal decomposition of Theorem \ref{A-semiorth-dec-thm} together with the adjunction
of the induction and restriction functors 
$(i_{\BB^{\circ}}^{\BB^{\circ\circ}},r^{\BB^{\circ\circ}}_{\BB^{\circ}})$.

To check that $M\{[i,j],[j+1,m]\}$ belongs to $\ker(r^{\BB^{\circ\circ}}_{\BB^{\circ}})$, we need to prove
the surjectivity of the map
$$(P_{[i,m]})_{[1,r]}\to (P_{[j+1,m]})_{[1,r]}$$ 
induced by $x_{[i,j]}$. In other words, we need to check that the map
$$(x_{[1,r]\setminus [i,m]})/(x_{[1,r]})\rTo{x_{[i,j]}} (x_{[1,r]\setminus [j+1,m]})/(x_{[1,r]})$$
is surjective for every $m$. Indeed, if $j\ge r$ then the target is zero. Otherwise $j<r$ and the assertion
follows from the equality $x_{[1,r]\setminus [j+1,m]}=x_{[i,j]}x_{[1,r]\setminus [i,m]}$.

By Theorem \ref{A-semiorth-dec-thm}, the subcategory $\ker(r^{\BB^{\circ\circ}}_{\BB^{\circ}})$
is generated by the images of $\Phi_2,\ldots,\Phi_n$. Hence, it is generated by the images of projective modules
under $\Phi_2,\ldots,\Phi_n$. Since  these images are given by formula \eqref{Phi-j-main-formula},
the last assertion follows.

\noindent
(ii) The restriction functor between abelian categories
$$r^{\BB^{\circ\circ}}_{R}: \BB^{\circ\circ}-\mod\to R-\mod$$
is exact. Furthermore, by Lemma \ref{res-ind-lem}(i) the induction functor
$i^{\BB^{\circ\circ}}_{R}$ provides a cosection, so the category $R-\mod$ is a colocalization of
$\BB^{\circ\circ}-\mod$, which gives an equivalence
$$R-\mod\simeq \BB^{\circ\circ}-\mod/(\ker(r^{\BB^{\circ\circ}}_R)\cap\BB^{\circ\circ}-\mod).$$
The similar equivalence with derived categories follows from the work \cite{miyachi}.

As in (i), the assertion that $M\{I,J\}$ belongs to $\ker(r^{\BB^{\circ\circ}}_R)$
amounts to the surjectivity of the map
$$(x_{[1,k]\setminus (I\sqcup J)})\rTo{x_I} (x_{[1,k]\setminus J}),$$
which follows from the equality $x_{[1,k]\setminus J}=x_Ix_{[1,k]\setminus (I\sqcup J)}$.

Let $\CC\sub {\BB^{\circ}}-\mod$ denote the kernel of $r^{\BB^{\circ}}_R$, i.e., the full subcategory consisting of 
${\BB^{\circ}}$-modules $M$ such that $M_{[1,k]}=0$. It is enough to prove that the modules $r^{\BB^{\circ\circ}}_{\BB^{\circ}}(M\{[j],[i,j-1]\})$ generate
$\CC$ as an abelian category. We will use the induction on $k$. In the case $k=1$, we have $\CC=0$ so there is nothing
to prove.

Let $\DD\sub {\BB^{\circ}}-\mod$ denote the full subcategory of ${\BB^{\circ}}$-modules $M$ such that $M_1=0$.
Let us set ${\ov{\BB}^{\circ\circ}}=\BB^{\circ\circ}_{[2,k]}[x_1]$, ${\ov{\BB}^{\circ}}=\BB^{\circ}_{[2,k]}[x_1]$.
Note that the restriction functor $r^{\BB^{\circ}}_{{\ov{\BB}^{\circ}}}:{\BB^{\circ}}-\mod\to{\ov{\BB}^{\circ}}-\mod$ 
induces an equivalence
$$\Phi:\DD\simeq {\ov{\BB}^{\circ}}-\mod.$$ 
Furthermore, as we have seen in Lemma \ref{res-ind-lem}(ii), we
have $r^{\BB^{\circ\circ}}_{\BB^{\circ}} i^{\BB^{\circ\circ}}_{{\ov{\BB}^{\circ\circ}}}(M)\in \DD$ and 
$$\Phi(r^{\BB^{\circ\circ}}_{\BB^{\circ}} i^{\BB^{\circ\circ}}_{{\ov{\BB}^{\circ\circ}}}(M))=r^{{\ov{\BB}^{\circ\circ}}}_{{\ov{\BB}^{\circ}}}(M).$$ 

Let $\CC'\sub {\ov{\BB}^{\circ}}-\mod$ denote the kernel of $r^{{\ov{\BB}^{\circ}}}_{R[x_1]}$. 
By induction assumption, $\CC'$ is generated by the modules $r^{{\ov{\BB}^{\circ\circ}}}_{{\ov{\BB}^{\circ}}}(\ov{M}\{[j],[i,j-1]\})$, where $2\le i<j\le k$, and
$\ov{M}\{[j],[i,j-1]\}$ are ${\ov{\BB}^{\circ\circ}}$-modules defined similarly to $M\{[j],[i,j-1]\}$. We have $i^{\BB^{\circ\circ}}_{{\ov{\BB}^{\circ\circ}}}\ov{M}\{[j],[i,j-1]\}=M\{[j],[i,j-1]\}$ for $2\le i<j\le k$,
so  
$$\Phi^{-1}(r^{{\ov{\BB}^{\circ\circ}}}_{{\ov{\BB}^{\circ}}}(\ov{M}\{[j],[i,j-1]\}))=r^{\BB^{\circ\circ}}_{\BB^{\circ}} i^{\BB^{\circ\circ}}_{{\ov{\BB}^{\circ\circ}}}(\ov{M}\{[j],[i,j-1]\})=r^{\BB^{\circ\circ}}_{\BB^{\circ}}(M\{[j],[i,j-1]\}).$$
Thus, the subcategory $\Phi^{-1}(\CC')=\CC\cap\DD$ belongs to the subcategory generated by the modules $r^{\BB^{\circ\circ}}_{\BB^{\circ}}(M\{[j],[i,j-1]\})$, with $2\le i<j\le k$.

It remains to check that $\CC$ is generated by $\CC\cap\DD$ and by the modules $r^{\BB^{\circ\circ}}_{\BB^{\circ}}(M\{[j],[1,j-1]\})$.
Note that $\CC\cap\DD$ is precisely the kernel of the functor
$$r:\CC\to R/(x_1,x_2\ldots x_k)-\mod: M\mapsto M_1.$$
Now for $2\le j\le k$, let us define the functor 
$$i_*(j):R/(x_1,x_j)\to \CC$$
sending $N$ to the ${\BB^{\circ}}$-module $M$ such that
$$M_{[1,m]}=\begin{cases} N & m<j, \\ 0 & m\ge j,\end{cases}$$
with the maps $M_{[1,m-1]}\to M_{[1,m]}$ (resp., $M_{[1,m]}\to M_{[1,m-1]}$) being the identity maps 
(resp., multiplication by $x_m$) for $m<j$.

For $2\le j\le k$, let $\CC_j\sub \CC$ denote the full subcategory of $M\in\CC$ such that $x_jM_1=0$.
Since for every $M\in\CC$ one has $x_2\ldots x_kM_1=0$, it is easy to see that the subcategories $\CC_j$ generate $\CC$
as an abelian category. 
Now for every $M\in\CC_j$, there is a natural morphism 
$$i_*(j)(M_1)\to M$$
inducing the identity map $i_*(j)(M_1)_1\to M_1$.
Hence, the kernel and the cokernel of the morphism $i_*(j)(M_1)\to M$ are in $\CC\cap\DD$.
It follows that the images of the functors $i_*(j)$ together with $\CC\cap\DD$ generate $\CC$.

It remains to observe that for each $j\ge 2$, there is an exact sequence
$$0\to i_*(j)(R/(x_1,x_j))\to r^{\BB^{\circ\circ}}_{\BB^{\circ}} M\{[j],[1,j-1]\}\to r^{\BB^{\circ\circ}}_{\BB^{\circ}} M\{[j],[2,j-1]\}\to 0$$
(where for $j=2$ the last term is zero).  
This implies that the image of $i_*(j)$ is contained in the subcategory generated by $r^{\BB^{\circ\circ}}_{\BB^{\circ}} M\{[j],[1,j-1]\}$ and 
$r^{\BB^{\circ\circ}}_{\BB^{\circ}} M\{[j],[2,j-1]\}$. 
\ed

          

\subsection{Cases $k=2$ and $k=3$}
          
Consider first the case $k=2$. We have $R_{[1,2]}=\k[x_1,x_2]/(x_1x_2),$ 
$${\BB^{\circ\circ}_{[1,2]}}=\End_R(R\oplus R/(x_1)\oplus R/(x_2)), \ \ {\BB^{\circ}_{[1,2]}}=\End_R(R\oplus R/(x_1)).$$
Note that we have an isomorphism 
$$R/(x_1)\oplus R/(x_2)\simeq I,$$
where $I=(x,y)\sub R$ is the ideal corresponding to the node. Thus, ${\BB^{\circ\circ}_{[1,2]}}$ is isomorphic to the Auslander order over the affine nodal
curve $\Spec(R_{[1,2]})$. We can also think of ${\BB^{\circ\circ}_{[1,2]}}$ as the algebra of the quiver with relations given in Figure \ref{quiverpop},
where the vertex $L_0$ (resp., $L_2$) corresponds to $[2]$ (resp., $[1]$).

The semiorthogonal decomposition of Proposition \ref{B-semiorth-dec-prop} becomes
$$\Perf({\BB^{\circ}_{[1,2]}})=\lan \lan \ov{P}_2\ran, \lan P_1\ran\ran,$$
where $\lan \ov{P_2}\ran\simeq \Perf(\k[x_1])$, $\lan P_1\ran\simeq \Perf(\k[x_2])$. Furthermore, $\Ext^*(\ov{P}_2,P_1)\simeq \k[-1]$.
 
The semiorthogonal decomposition of Theorem \ref{A-semiorth-dec-thm} becomes
$$\Perf({\BB^{\circ\circ}_{[1,2]}})=\lan \lan M\{[1][2]\}\ran, i^{\BB^{\circ\circ}}_{\BB^{\circ}}\Perf({\BB^{\circ}_{[1,2]}})\ran,$$
where $\lan M\{[1][2]\}\ran\simeq \Perf(\k)$. Furthermore, it is easy to see that $M\{[1][2]\}$ corresponds to the simple module over
the quiver in Figure \ref{quiverpop} at the vertex $L_0$.
 
The localization functor $D^b({\BB^{\circ\circ}_{[1,2]}})\to D^b(R_{[1,2]})$ takes $P_{[1,2]}$ to $R$ and $P_{[i]}$ to $R/(x_i)$. The kernel of the functor is
generated by the modules $M\{[1][2]\}$ and $M\{[2][1]\}$, which are precisely the simple modules at vertices $L_0$ and $L_2$ of the quiver in Figure \ref{quiverpop}.

Next, let us consider the case $k=3$. We have $R=R_{[1,3]}=\k[x_1,x_2,x_3]/(x_1x_2x_3)$,
$${\BB^{\circ}_{[1,3]}}=\End_R(R\oplus R/(x_1x_2)\oplus R/(x_1)),$$
while the algebra ${\BB^{\circ\circ}_{[1,3]}}$ can be described by the quiver with relations over $R$ given in Figure \ref{figure2},
where the vertices $L_2\times L_3$, $L_0\times L_3$ and $L_0\times L_1$ correspond to $[1]$, $[2]$ and $[3]$, $L_1\times L_3$ and $L_0\times L_2$
correspond to $[1,2]$ and $[2,3]$, $L_1\times L_2$ corresponds to $[1,3]$.
The algebra ${\BB^{\circ}_{[1,3]}}$ is described by the subquiver with vertices $[1]$, $[1,2]$ and $[1,3]$.

The semiorthogonal decomposition of Proposition \ref{B-semiorth-dec-prop} becomes
$$\Perf({\BB^{\circ}_{[1,3]}})=\lan \Perf(\k[x_2,x_3]), \Perf(\k[x_1,x_3]), \Perf(\k[x_2,x_3])\ran.$$
 
The semiorthogonal decomposition of Theorem \ref{A-semiorth-dec-thm} becomes
$$\Perf({\BB^{\circ\circ}_{[1,2]}})=\lan \lan M\{[2][3]\}\ran, \lan M\{[1][2,3]\}, M\{[1][2]\}\ran, i^{\BB^{\circ\circ}}_{\BB^{\circ}}\Perf({\BB^{\circ}_{[1,3]}})\ran,$$
where $\lan M\{[2][3]\}\ran\simeq \Perf(\k[x_1])$ and
$\lan M\{[1][2,3]\}, M\{[1][2]\}\ran\simeq \Perf({\BB^{\circ}_{[2,3]}})$.
The kernel of the localization $D^b({\BB^{\circ}_{[1,3]}})\to D^b(R_{[1,3]})$ is generated by the modules
$i_*(2)(R/(x_1,x_2))$, $i_*(3)(R/(x_1,x_3))$ and $r^{\BB^{\circ\circ}}_{\BB^{\circ}}M\{[3][2]\}$,
where the first module is $R/(x_1,x_2)$ supported at the vertex $[1]$, the last is $R/(x_2,x_3)$ supported at the vertex $[1,2]$.
The module $i_*(3)(R/(x_1,x_3))$ is supported at vertices $[1]$ and $[1,2]$, with both components given by $R/(x_1,x_3)$, the identity map from $[1,2]$ to $[1]$, 
and the multiplication by $x_2$ map from $[1]$ to $[1,2]$.

\section{Homological mirror symmetry}\label{comparison-sec}

\subsection{Isomorphism of algebras corresponding to two stops}

Now we specialize to the case $n=k-1$ on the A-side and the $\Z$-grading on the Fukaya category such that
$\deg(x_i)=0$ (see Theorem \ref{A-side-morphisms-general-thm}). The main observation is that in this case the algebra $\cA^{\circ\circ}$ on the A-side
is isomorphic to the algebra $(\BB^{\circ\circ})^{op}$.

First, we observe that there is the following bijection between the Lagrangians $(L_S)$ in this case and
the subsegments $I\sub [1,k]$:
\begin{equation}\label{Lagr-intervals-correspondence}
L_{[0,k]\setminus\{i,j\}}\leftrightarrow [i+1,j],
\end{equation}
where $0\le i<j\le k$.

\begin{lem}\label{mirror-lem} 
For $n=k-1$ 
one has an isomorphism of $R$-algebras
$$\cA^{\circ\circ}\simeq \BB^{\circ\circ}\simeq (\BB^{\circ\circ})^{op}=\bigoplus_{I,J\sub [1,k]}\Hom_{\BB^{\circ\circ}}(P_I,P_J),$$
compatible with the correspondence
\eqref{Lagr-intervals-correspondence}.
\end{lem}

\Pf . We have
$$\Hom_{\BB^{\circ\circ}}(P_I,P_J)\simeq\Hom_R(R/(x_J),R/(x_I))\simeq (x_{I\setminus J})/(x_I)=
x_{I\setminus J}\cdot R/(x_{I\cap J}).$$
It is easy to see that mapping $\Hom_{\BB^{\circ\circ}}(P_I,P_J)$ to
$\Hom_{\BB^{\circ\circ}}(P_J,P_I)$ using these identifications gives an isomorphism of the algebra
$\BB^{\circ\circ}$ with its opposite algebra (see below for the computation of compositions).

On the other hand, for $S=[0,k]\setminus\{a,b\}$ and $S'=[0,k]\setminus\{a',b'\}$ we have the
following four possibilities, and we can calculate the subintervals $I_\bullet$, $J_\bullet$ of
Corollary \ref{pair-partition-cor} in each of them.

\noindent
{\it Case 1}. $a<a'$, $b<b'$. If $a'\ge b$ then $S$ and $S'$ are not close, so we can assume that $a'<b$.
In this case we have $J_1=[a,a']$, $J_2=[b,b']$, so using Theorem \ref{A-side-morphisms-general-thm}(ii) we get
\begin{align*}
&\cA(S,S')=\k[x_1,\ldots,x_a]\ot \k[x_{a+1},\ldots,x_{a'}]\ot \k[x_{a'+1},\ldots,x_b]/(x_{[a'+1,b]})\\
&\ot \k[x_{b+1},\ldots,x_b'] \ot \k[x_{b'+1},\ldots,x_k] \cdot f_{S,S'} = 
R/(x_{[a'+1,x_b]})\cdot f_{S,S'}.
\end{align*}

\noindent
{\it Case 2}. $a<a'$, $b\ge b'$. In this case we have $J_1=[a,a']$, $I_1=[b',b]$ ($I_1$ should be omitted if $b=b'$), so
by Theorem \ref{A-side-morphisms-general-thm}(ii) we get
\begin{align*}
&\cA(S,S')=\k[x_1,\ldots,x_a]\ot \k[x_{a+1},\ldots,x_{a'}]\ot \k[x_{a'+1},\ldots,x_{b'}]/(x_{[a'+1,b']})\\
&\ot \k[x_{b'+1},\ldots,x_b]\ot \k[x_{b+1},\ldots,x_k] \cdot f_{S,S'} = 
R/(x_{[a'+1,x_{b'}]})\cdot f_{S,S'}.
\end{align*}

\noindent
{\it Case 3}. $a\ge a'$, $b<b'$. In this case we have $I_1=[a',a]$, $J_1=[b,b']$ ($I_1$ should be omitted if $a=a'$), so
by Theorem \ref{A-side-morphisms-general-thm}(ii), we get
\begin{align*}
&\cA(S,S')=\k[x_1,\ldots,x_{a'}]\ot \k[x_{a'+1},\ldots,x_{a}]\ot \k[x_{a+1},\ldots,x_{b}]/(x_{[a+1,b]})\\
&\ot \k[x_{b+1},\ldots,x_{b'}]\ot \k[x_{b'+1},\ldots,x_k] \cdot f_{S,S'} = 
R/(x_{[a+1,x_{b}]})\cdot f_{S,S'}.
\end{align*}

\noindent
{\it Case 4}. $a\ge a'$, $b\ge b'$. If $a\ge b'$ then $S$ and $S'$ are not close, so we can assume that $a<b'$.
In this case we have $I_1=[a',a]$, $I_2=[b',b]$ (where $I_1$ is omitted if $a=a'$ and $I_2$ is omitted if $b=b'$).
By Theorem \ref{A-side-morphisms-general-thm}(ii), we get
\begin{align*}
&\cA(S,S')=\k[x_1,\ldots,x_{a'}]\ot \k[x_{a'+1},\ldots,x_{a}]\ot \k[x_{a+1},\ldots,x_{b'}]/(x_{[a+1,b']})\\
&\ot \k[x_{b'+1},\ldots,x_{b}]\ot \k[x_{b+1},\ldots,x_k] \cdot f_{S,S'} = 
R/(x_{[a+1,x_{b'}]})\cdot f_{S,S'}.
\end{align*}

In all of these cases we deduce that 
$$\cA(S,S')\simeq R/(x_{I\cap J})\cdot f_{S,S'},$$
where  $I=[a+1,b]$, $J=[a'+1,b']$.
So we get an identification of $R$-modules 
$$\cA(S,S')\simeq\Hom(P_I,P_J)$$
sending $f_{S,S'}$ to $x_{I\setminus J}$.

To check the compatibility with the composition, we note that for three intervals $I,J,K\sub [1,k]$ one has
$$x_{I\setminus J}\cdot x_{J\setminus K}=x_{I\setminus K}\cdot \prod_{i\in (J\setminus(I\cup K))\cup (I\cap K\setminus J)}x_i.$$
Thus, the assertion follows from Theorem \ref{A-side-morphisms-general-thm}(iii) and the equality
\begin{equation}\label{comb-triple-eq}
\{i:\ |\ [i-1,i]\sub T(S,S',S'')\}=(J\setminus(I\cup K))\cup (I\cap K\setminus J),
\end{equation}
where $L_S$, $L_{S'}$ and $L_{S''}$ correspond to $I$, $J$ and $K$ under \eqref{Lagr-intervals-correspondence}.
The proof of \eqref{comb-triple-eq} is a straightforward but tedious check, so we will consider only one of the cases.
Let 
\begin{align*}
&S=[0,k]\setminus\{a,b\}, \ \ S'=[0,k]\setminus\{a',b'\}, \ \ S''=[0,k]\setminus\{a'',b''\}, \\ 
&I=[a+1,b], \ \ J=[a'+1,b'], \ \ K=[a''+1,b''].
\end{align*}
Assume that $a<a'$, $b<b'$, $a''<a'$, $b''<b'$. Then we have $J_1=[a,a']$, $J_2=[b,b']$, $I'_1=[a'',a']$, $I'_2=[b'',b']$,
hence,
$$T(S,S',S'')=[\max(a,a''),a'] \cup [\max(b,b''),b'].$$
On the other hand, 
$$I\cap K\setminus J=[\max(a+1,a''+1),a'], \ \ J\setminus(I\cup K)=[\max(b,b'')+1,b'],$$
so the equality \eqref{comb-triple-eq} follows in this case.
The other cases are considered similarly.
\ed

\subsection{Equivalences of categories}

\begin{thm}\label{main-thm} 
There are equivalences of enhanced triangulated categories over $\k$,
\begin{equation}\label{2stop-A-B-equivalence}
\WW(\PP_{k-1},\La_Z)\simeq \Perf(\BB^{\circ\circ}_{[1,k]}),
\end{equation}
\begin{equation}\label{1stop-A-B-equivalence}
\WW(\PP_{k-1},\La_{1})\simeq \Perf(\BB^{\circ}_{[1,k]}),
\end{equation}
Assume that the base ring $\k$ is regular. Then we also have an equivalence
\begin{equation}\label{0stop-A-B-equivalence}
\WW(\PP_{k-1})\simeq D^b(R_{[1,k]}),
\end{equation}
The above equivalences are compatible with localization functors on the A-side and the restriction functors $r^{\BB^{\circ\circ}}_{\BB^{\circ}}$ and $r^{\BB^{\circ}}_R$.
Furthermore, the first equivalence sends the Lagrangian
$L_S=\prod_{a\in S} L_a$ for $S=[0,k]\setminus\{i,j\}$, where $i<j$, to the projective module $P_{[i+1,j]}$.
Here the Fukaya categories are equipped with the $\Z$-grading coming from a unique grading structure on 
$\PP_{k-1}=M_{k-1,k}$ such that $\deg(x_i)=0$ for $i=1,\ldots,k$ (see Theorem \ref{A-side-morphisms-general-thm}).
\end{thm}

\Pf . Note that by Theorem \ref{A-side-morphisms-general-thm}(iv), shifting the graded structures on each $L_S$ if needed,
we can achieve that the algebra $\cA^{\circ\circ}$ is concentrated in degree $0$.
Since the Lagrangians $(L_S)$ generate the wrapped Fukaya category,
the first equivalence is an immediate consequence of the isomorphism of algebras proved in Lemma \ref{mirror-lem}.

The second and third equivalences are obtained from the first by localization, using Theorem \ref{A-localization-thm} 
on the A-side and
Theorem \ref{B-loc-thm} on the B-side. Namely, we observe that under the equivalence
\eqref{2stop-A-B-equivalence}, the resolutions of objects $T_2\times X$ (resp., of $T_1\times X$)
from Proposition \ref{stop-resolutions-prop}
(in the case $n=k-1$) correspond to the complexes defining modules $M\{[i,j],[j+1,m]\}$ (resp., $M\{[j+1,m],[i,j]\}$).
Thus, under \eqref{2stop-A-B-equivalence} the subcategory 
$\DD_2\sub \WW(M_{k-1,k})$ (resp., $\lan\DD_1,\DD_2\ran$) 
corresponds to the subcategory $\ker(r^{\BB^{\circ\circ}}_{\BB^{\circ}})\sub D^b(\BB^{\circ\circ})$
generated by all the modules $M\{[i,j],[j+1,m]\}$ (resp., $\ker(r^{\BB^{\circ\circ}}_R)$ generated by all $M\{I,J\}$).
Hence, using Theorems \ref{A-localization-thm} and \ref{B-loc-thm}, we derive the equivalences
\eqref{1stop-A-B-equivalence} and \eqref{0stop-A-B-equivalence}.
\ed

\begin{rmk}  It follows from Remark \ref{ov-P-res-rem} that
under equivalence \eqref{1stop-A-B-equivalence} of Theorem \ref{main-thm}, 
 the object $\overline{P}_{i}$ defined in Proposition \ref{B-semiorth-dec-prop} (involved in a semi-orthogonal decomposition of $D^b(\mathcal{B}^\circ_{[1,k]})$) 
 corresponds to $L_S = \prod_{a \in S} L_a$ with $S = [0,k] \setminus \{i-1,i\}$.
  \end{rmk}

\begin{rmk} The semiorthogonal decomposition 
$$D^b(\BB^{\circ\circ})=\lan \im(\Phi_k),\ldots,\im(\Phi_2),\im(\Phi_1)\ran$$ 
constructed in Theorem \ref{A-semiorth-dec-thm} has a natural geometric meaning on the A-side.
Namely, let $\CC_i\sub \WW(M_{k-1,k},\La_Z)$ be the subcategory corresponding to $\im(\Phi_i)$.
Then $\CC_1$ is generated by the Lagrangians $L_S$ with $S\sub [1,k]$, and for each $i=2,\ldots,k$,
the subcategory $\CC_i$ is generated by the Lagrangians of the form $T_2\times L_{[0,i-3]\sqcup S'}$ with $S'\sub [i,k]$,
$|S'|=k-i$. The required semiorthogonalities can be easily seen on the A-side. Also, the natural functor 
$$\CC_1\to \WW(M_{k-1,k},\La_1)$$ 
is an equivalence, which corresponds to the equivalence of $\im(\Phi_1)$ with $D^b(\BB^{\circ})$.
On the other hand, for each $i=2,\ldots,k$, there is an equivalence
$$\CC'_i\rTo{\sim}\CC_i: L\mapsto T_2\times L$$
from the subcategory $\CC'_i\sub \WW(M_{k-2,k},\La_Z)$ generated by all $L_{[0,i-3]\sqcup S'}$ with $S'\sub [i,k]$.
Furthermore, since $\End(L_{[0,i-3]})\simeq k[x_1,\ldots,x_{i-2}]$, one can see that 
there is an equivalence
$$\CC''_i\otimes_k k[x_1,\ldots,x_{i-2}]\rTo{\sim} \CC'_i: L\mapsto L_{[0,i-3]}\times L,$$ 
where $\CC''_i\sub \WW(M_{k-i,k},\La_Z)$ is generated by $L_{S'}$ with $S'\sub [i,k]$. Finally,
we have a natural equivalence
$$\CC''_i\rTo{\sim} \WW(M_{k-i,k-i+1},\La_1)$$
induced by filling in the first $i-1$ interior holes and removing one of the stops.
\end{rmk}

\subsection{Homological mirror symmetry for abelian covers}

We are going to deduce from Theorem \ref{main-thm}, a similar equivalence involving wrapped Fukaya categories
of finite abelian covers of $M_{k-1,k}=\PP_{k-1}$.

Namely, as is well known, for $k>2$, there is a natural isomorphism $\pi_1(\PP_{k-1})\simeq \Z^k$,
where as generators we can take products of loops $\ga_i=u_iv_i$, $i=1,\ldots,k$, in $\Sigma$ with $k-2$ points in $\Sigma$.
We denote these generators as $(\hat{\ga}_i)$.

Now let us fix a homomorphism
$$\phi:\pi_1(\PP_{k-1})\simeq \Z^k\to \Gamma,$$
where $\Gamma$ is a finite abelian group.
It is determined by the $k$ elements
$$\phi_i:=\phi(\hat{\ga}_i)\in \Gamma, \ i=1,\ldots,k.$$
Let 
$$\pi:M\to\PP_{k-1}$$
be the finite covering associated with $\phi$, so that $\Gamma$ acts on $M$ as a group of deck transformations.

Then $M$ has the induced symplectic structure, the induced stops $\pi^{-1}(\La)$ and the induced grading structure,
so we can consider the category $\WW(M,\pi^{-1}(\La))$.


Note that each of the Lagrangians $L_S\sub \PP_{k-1}$, for $S=[0,k]\setminus\{i,j\}$ can be lifted to
a connected Lagrangian $\wt{L}_S\sub M$ (since it is simply connected). Let us fix one such lift $\wt{L}_S$ 
for each $S$ (we'll have to adjust it later). Note that all other connected lifts of $L_S$ are of the form $\ga\wt{L}_S$
for some $\gamma\in \Gamma$, where $\Gamma$ acts on $M$ by deck transformations.


As explained in \cite[Sec.\ 8b]{Seidel-quartic}, 
a choice of liftings $(\wt{L}_S)$ defines a $\Gamma$-grading on each space $\hom(L_S,L_{S'})$. 
Namely, a Reeb chord $x:[0,1]\to \PP_{k-1}$ with $x(0)\in L_S$, $x(1)\in L_{S'}$ lifts to a path $\wt{x}:[0,1]\to M$,
such that $\wt{x}(0)\in \wt{L}_S$ and $\wt{x}(1)\in \gamma \wt{L}_{S'}$, for some $\gamma\in \Gamma$, and we set
$\deg^A_\Gamma(x)=\gamma$.

This gives a certain $\Gamma$-grading on the algebra $\cA^{\circ\circ}$.
Note that $\Gamma$-degrees of the elements $x_i\in \Endlc(L_S)$ are just given by
$$\deg^A_\Gamma(x_i)=\phi_i.$$

Note that endomorphism algebra of the objects $(\gamma\wt{L}_S)$ is obtained from the $\Gamma$-grading on
$\cA^{\circ\circ}$ by the following formal construction: it is given by
\begin{equation}\label{A-Gamma-graded-eq}
\bigoplus_{\ga,\ga'\in\Gamma} \cA^{\circ\circ}_{\ga'-\ga},
\end{equation}
with the natural product. 

Let 
$$G:=\Gamma^*$$ be the finite commutative group scheme dual to $\Gamma$ (so if $\Gamma=\prod \Z/d_i\Z$
then $G=\prod \mu_{d_i}$). Then $\Gamma$ can be identified with the set of all algebraic characters of $G$.
Suppose $A$ is an associative algebra with a $\Gamma$-grading. Then we can form a new associative algebra
similarly to \eqref{A-Gamma-graded-eq} by setting
$$A_{\Ga}:=\bigoplus_{\ga,\ga'\in\Gamma} A_{\ga'-\ga}.$$
On the other hand, we can use the $\Gamma$-grading to define an algebraic action of $G$ on $A$ and consider the
category of $A$-modules with compatible $G$-action, or equivalently, $\Gamma$-graded $A$-modules.
For every $\ga\in\Gamma$ we have a natural such module $A\ot\ga$,
and we have an identification
$$A_{\Ga}\simeq \bigoplus_{\ga,\ga'} \Hom_A(A\ot\ga,A\ot\ga')^G.$$
Thus, the category described by the algebra $A_{\Ga}$ is equivalent to the perfect derived category of $A$-modules
with $G$-action (one can also explicitly identify $A_{\Ga}$ with the crossed product ring $A[G]$).

\begin{prop} The collection of Lagrangians $\gamma\tilde{L}_S$ for $\gamma \in
\Gamma$ and $S = [0,k] \setminus \{i,j\}$, $i<j$ generates the wrapped Fukaya
category $\mathcal{W}(M, \pi^{-1}(\Lambda_Z))$.  \end{prop}

\begin{proof} The proof follows along the lines of \cite{aurouxicm} where Auroux proves generation of $\mathcal{W}(\mathcal{P}_{k-1}, \Lambda_Z)$ by the Lagrangians $L_S$. Namely, one first constructs a simple branched $m$-fold covering of the unit disk, a Lefschetz fibration with zero-dimensional fibers, $\varpi : \Sigma \to \mathbb{D}$ for some $m\geq k$ and such that the collection of Lagrangians $L_i$ for $i=0,\ldots, k$ is a subset of a distinguished set of thimbles $\{\Delta_i\}_{i=0}^{m}$ for $\varpi$ and $Z$ is a subset of $\varpi^{-1}(-1)$. Note that, by re-indexing if necessary, we can assume that $L_i = \Delta_i$ for $i=0,\ldots k$. The fact that such a simple branched covering exists is elementary (see \cite[Sec. 3.3.4]{aurouxicm}). A useful observation here is that the number of stops in $\varpi^{-1}(-1) \setminus Z$ is the same as $m-k$. Furthermore, localizing along points in $\varpi^{-1}(-1) \setminus Z$ corresponds to removing extra objects in the collection $\{ \Delta_i \}_{i=0}^m$ to finally reduce it to the collection $\{ L_i \}_{i=0}^k$.

Next, Auroux associates a Lefschetz fibration $\varpi_n : \Sym^n(\Sigma) \to \mathbb{D}$ whose thimbles are given by the collection of products Lagrangians $\Delta_S$ where $S$ runs through subsets of $[0,m]$ of size $n$, and applies a celebrated result of Seidel \cite[Thm. 18.24] {seidelbook} to show that this collection generates the wrapped Fukaya category $\mathcal{W}(\Sym^n(\Sigma), \Lambda_{\varpi^{-1}(-1)})$. Finally, to deduce generation of $\mathcal{W}(\Sym^n(\Sigma), \Lambda_Z)$ by the collection of products of Lagrangian $L_S$ where $S$ runs through subsets of $[0,k]$ of size $n$, Auroux applies the localization functors associated to removing the extra stops in $\Lambda_{\varpi^{-1}(-1)} \setminus \Lambda_Z$. Namely, localizing along the extra stops $\Lambda_{\varpi^{-1}(-1) \setminus Z}$ corresponds to removing extra objects in the collection $\{ \Delta_S \}$ for $S \subset [0,m]$ to finally reduce it to the collection $L_S$ for $S \subset [0,k]$. This is thanks to the existence of basic exact triangles recalled in equation \eqref{exacttri} associated to exact triangles in $\mathcal{W}(\Sigma, Z)$.  

The key observation that allows us to apply the same strategy to generate $\mathcal{W}(M, \pi^{-1}(\Lambda_Z))$ is that the composition 
\[ \varpi_{k-1} \circ \pi : M \to \mathbb{D} \]
is again a Lefschetz fibration and the stops $\pi^{-1}(\Lambda_Z)$ is a subset of $(\varpi_{k-1} \circ \pi)^{-1}(-1)$. Moreover, the thimbles of $\varpi_{k-1} \circ \pi$ are given by lifts $\{ \gamma \tilde{\Delta}_S \}_{\gamma, S}$ of thimbles $\Delta_S$ after fixing connected lifts $\tilde{\Delta}_S$ of $\Delta_S$. Thus, again by Seidel's theorem, the set of these lifted thimbles generate the partially wrapped Fukaya category $\mathcal{W}(M, (\varpi_{k-1} \circ \pi)^{-1}(-1))$. 

Finally, removing stops in $\varpi_{k-1}^{-1}(-1) \setminus Z$ corresponds to localization
of lifts of the extra thimbles $\{\Delta_S\}$ that are not in the collection $\{ L_S \}$. To justify this we apply Lemma \ref{nextlemma} which completes the proof.
\end{proof}

\begin{lem} \label{nextlemma} Suppose that $L_1, L_1', L_1'' \subset \Sigma$ such that $L_1''$ is the arc obtained by sliding $L_1$ along $L_1'$ avoiding a set of stops $Z$ so that we have an exact triangle
\[  L_1 \to L_1' \to L_1'' \to L_1[1] \]
in $\mathcal{W}(\Sigma, Z)$. Let $L_2, \ldots, L_{k-1}$ be disjoint arcs in $\Sigma$ 
and let $L_S = L_1 \times L_2 \times \ldots L_{k-1}$, $L'_S = L_1' \times L_2 \times \ldots L_{k-1}$ and $L''_S = L_1'' \times L_2 \times \ldots L_{k-1}$ be Lagrangians in $\mathcal{P}_{k-1}$ so that we have the corresponding exact triangle
\begin{equation}\label{LS-exact-triangle-lift-eq} 
L_S \to L_S' \to L_S'' \to L_S[1] 
\end{equation}
in $\mathcal{W}(\Sigma, \Lambda_Z)$.
Let $\pi : M \to \mathcal{P}_{k-1}$ be a covering and let $\tilde{L}_S$, $\tilde{L}'_S$ and $\tilde{L}''_S$ be connected lifts of these Lagrangians to $M$. Then, in $\mathcal{W}(M, \pi^{-1}(\Lambda_Z))$, the Lagrangian $\tilde{L}_S$ is generated by the collection $\{ \gamma \tilde{L}'_S \}_{\gamma \in\Gamma}  \cup \{ \gamma \tilde{L}''_S \}_{\gamma \in \Gamma}$.  
\end{lem}
\begin{proof}
As we explained before, we can describe the subcategory of $\WW(M,\pi^{-1}(\Lambda_Z))$, generated by 
$(\gamma\wt{L}_S)$, by the algebra \eqref{A-Gamma-graded-eq}, and hence, identify it with
the perfect derived category of $\Gamma$-graded $\cA^{\circ\circ}$-modules, such that $\gamma\wt{L}_S$ corresponds
to $P_S\ot\gamma$, for the natural projective  $\cA^{\circ\circ}$-modules $P_S$

Now we observe that all the morphisms in the exact triangle \eqref{LS-exact-triangle-lift-eq} are homogeneous with respect to the 
$\Gamma$-grading (i.e., they all have some degree in $\Gamma$). This implies that this exact triangle can be
viewed as an exact triangle of $\Gamma$-graded $\cA^{\circ\circ}$-modules of the form
$$P_S\to P_{S'}\ot \gamma'\to P_{S''}\ot \gamma''\to P_S[1].$$
The corresponding exact triangle in $\WW(M,\pi^{-1}(\Lambda_Z))$ proves our assertion.
\end{proof} 

On the B-side, we can consider the homomorphism
$$G\to \G_m^k,$$ 
dual to the homomorphism $\Z^k\to \Gamma$.
The group $\G_m^k$ acts on the ring $R=R_{[1,k]}=\k[x_1,\ldots,x_k]/(x_1\ldots x_k)$ by rescaling the coordinates,
and on the $R$-modules $R/(x_I)$, so we get the induced action on the algebras $\BB^\circ=\BB^\circ_{[1,k]}$ and 
$\BB^{\circ\circ}=\BB^{\circ\circ}_{[1,k]}$.
Via the homomorphism $G\to\G_m^k$ we get algebraic $G$-actions on $R$, $\BB^\circ$ and $\BB^{\circ\circ}$.

Let $\Perf_G(\BB)$ (resp., $D^b_G(R)$)
denote the perfect derived category of $G$-equivariant $\BB$-modules (resp., bounded derived catogory of
finitely generated modules $G$-equivariant $R$-modules). 

\begin{thm}\label{abelian-cover-thm} 
We have equivalences of enhanced triangulated categories over $\k$,
\begin{equation}\label{2stop-Gamma-A-B-equivalence}
\WW(M,\pi^{-1}(\La))\simeq \Perf_G(\BB^{\circ\circ}_{[1,k]}),
\end{equation}
\begin{equation}\label{1stop-Gamma-A-B-equivalence}
\WW(M,\pi^{-1}(\La_{1}))\simeq \Perf_G(\BB^{\circ}_{[1,k]}).
\end{equation}
Assume that $\k$ is a regular ring.
Then we have an equivalence
\begin{equation}\label{0stop-Gamma-A-B-equivalence}
\WW(M)\simeq D^b_G(R_{[1,k]}),
\end{equation}
where the Fukaya categories are equipped with the $\Z$-grading induced by the grading structure on $\PP_{k-1}$ considered
in Theorem \ref{main-thm}.
\end{thm}

\begin{proof} 
The action of $G$ on $\BB^{\circ\circ}$ can also be viewed as a $\Gamma$-grading on it, $\deg^B_\Gamma$,
which satisfies 
$$\deg^B_\Gamma(x_i)=\phi_i.$$
By the isomorphism of Lemma \ref{mirror-lem}, it can be viewed as another $\Gamma$-grading on $\cA^{\circ\circ}$.
Applying Theorem \ref{A-side-morphisms-general-thm}(iv), we deduce the existence of $(\ga_S)$ such that
$$\deg^B_\Gamma(f_{S,S'})=\deg^A_\Gamma(f_{S,S'})+\ga_{S'}-\ga_S.$$
This means that by adjusting the choices of lifts $\wt{L}_S$ (by the action of $\ga_S$),
we can achieve that the gradings $\deg^B_\Gamma$ and $\deg^A_\Gamma$ are the same.

Now as we have observed before, the category $\Perf_G(\BB^{\circ\circ}_{[1,k]})$ 
is equivalent to the perfect derived category of the algebra
$$\BB^{\circ\circ}_{\Ga}=\bigoplus_{\ga,\ga'\in\Gamma} \BB^{\circ\circ}_{\ga'-\ga}.$$
Since the isomorphism between $\cA^{\circ\circ}$ and $\BB^{\circ\circ}$ is compatible with the $\Gamma$-gradings, the latter algebra is isomorphic to
\eqref{A-Gamma-graded-eq}, which implies the equivalence \eqref{2stop-Gamma-A-B-equivalence}.

  To deduce equivalences \eqref{1stop-Gamma-A-B-equivalence} and \eqref{0stop-Gamma-A-B-equivalence}
we first observe that an analog of Theorem \ref{B-loc-thm} holds for equivariant categories, where the modules
$M\{I,J\}$ are equipped with natural $G$-equivariant structure, and to generate the kernels of 
$r^{\BB^{\circ\circ}}_{\BB^{\circ}}$ and $r^{\BB^{\circ\circ}}_R$ we use appropriate modules twisted by all possible elements
of $\Gamma$. Thus, we derive these equivalences as in Theorem \ref{main-thm} using localization on both sides.
\end{proof}

As an application we deduce a version of homological mirror symmetry for invertible polynomials (see the recent paper \cite{LU} for background on this topic). Let 
\begin{align}
 \w = \sum_{i=1}^k \prod_{j=1}^k x_j^{a_{ij}}
\end{align}
be an \emph{invertible} polynomial, which is a weighted homogeneous polynomial that has an isolated critical point at 
the origin and is described by some integer matrix 
$
 A = (a_{ij})_{i, j=1}^k
$
with non-zero determinant.

Let 
$$M_{\w}:= \{ (x_1,\ldots,x_k)\in (\C^*)^{\times k} \ |\ \w(x_1,\ldots, x_k) = 1 \}$$ 
be the ``punctured'' Milnor fiber. We have a covering map 
\[ \pi : M_{\w} \to \mathcal{P}_{k-1} \]
given by $(x_1,x_2, \ldots, x_k) \to ( \prod_{j=1}^k x_j^{a_{1j}}, \prod_{j=1}^k x_j^{a_{2j}}, \ldots, \prod_{j=1}^k x_j^{a_{kj}})$ where we view $\mathcal{P}_{k-1}$ as a hypersurface in $(\C^*)^{\times k}$ via the identification
\[ \mathcal{P}_{k-1} = \{ (x_1,\ldots, x_k) \in (\C^*)^{\times k} : x_1 + x_2 + \ldots + x_k = 1 \}. \]
The group of deck transformations of this covering map is
\[ \Gamma = \{ (t_1,t_2,\ldots, t_k) \in \mathbb{G}_m^{\times k} : \forall i, t_1^{a_{i1}} t_2^{a_{i2}} \ldots t_k^{a_{ik}} = 1  \}, \]
which is exactly the group of diagonal symmetries of $\w$.

Let $G= Hom( \Gamma, \mathbb{G}_m)$ be the dual group. As an immediate corollary to Theorem \ref{abelian-cover-thm}, we get the following homological mirror symmetry statement for $M_{\w}$.

\begin{cor} Assume $\k$ is regular. Then $\WW(M_\w)\simeq D^b_G(R_{[1,k]})$. 
\end{cor}

\begin{rmk}\label{1dim-coverings-rem}
There is also an analog of Theorem \ref{abelian-cover-thm} for arbitrary finite
connected coverings of the $3$-punctured sphere $\Sigma=\PP_1$.
Namely, we have an identification of $\pi_1(\Sigma)$ with the free group with two generators $\ga_1$, $\ga_2$, so
such a covering $\pi:M\to \Sigma$ corresponds to a finite set $S$ with a pair of permutations $\si_1,\si_2\in\Aut(S)$,
generating a transitive action of $\pi_1(\Sigma)$ on $S$ (so that $\ga_i$ acts as $\si_i$). 
Choosing a lifting of the Lagrangians $L_0,L_1,L_2$
to the universal covering of $\Sigma$, we get as above, a $\pi_1(\Sigma)$-grading on the endomorphism algebra
$A$ of $(L_0,L_1,L_2)$ in the partially wrapped Fukaya category $\WW(\Sigma,Z,\eta)$ where $Z=q_1\cup q_2$
(see Figure \ref{pop}).
Now a similar reasoning to Theorem \ref{abelian-cover-thm} shows that $\WW(M,\pi^{-1}(Z),\eta)$ is equivalent
to the perfect derived category of $S$-graded modules over $A$ (where the $S$-grading on a module is compatible
with the $\pi_1(\Sigma)$-grading on $A$). Furthermore, we can choose liftings of the Lagrangians in such a way 
that the $\pi_1(\Sigma)$-grading on $A$, which is identified with the Auslander order over $\k[x_1,x_2]/(x_1x_2)$,
is induced by the grading $\deg(x_i)=\ga_i$, $i=1,2$. Localizing we get a similar connection between the fully wrapped
Fukaya category of $M$ and $S$-graded modules over $\k[x_1,x_2]/(x_1x_2)$.

For example, taking $\Ga=\Z/d$ (for $d>0$) and the homomorphism $\ga_1\mapsto 1\mod (d)$, $\ga_2\mapsto -1\mod(d)$, corresponds
to the $d:1$ covering $\wt{\Sigma}\to \Sigma$, that can be identified with the $(d+2)$-punctured sphere. Namely, if we realize $\Sigma$ with $\P^1\setminus \{0,1,\infty\}$
in such a way that the puncture with stops corresponds to $1$, then we can take the covering $\pi:\P^1\to \P^1:z\mapsto z^d$, ramified at $0$ and $\infty$, and
identify $\wt{\Sigma}$ with 
$$\P^1\setminus \pi^{-1}(\{0,1,\infty\})=\P^1\setminus \{0,\infty,(\exp(2\pi i m/d))_{m=0,\ldots,d-1}\}.$$
The above equivalence relates the partially wrapped Fukaya category of $\wt{\Sigma}$ (with respect to the preimages of two stops) with modules over the
Auslander order on the corresponding stacky curve, the quotient of $\Spec(\k[x_1,x_2]/(x_1x_2))$ by the action of $\mu_d$, the group of $d$th roots of unity,
where $\zeta$ acts on $(x_1,x_2)$ by $(\zeta x_1,\zeta^{-1}x_2)$. This equivalence is a particular case of \cite[Thm.\ A]{LP}.
\end{rmk}

\end{document}